\documentclass[paper=letter, fontsize=10pt,leqno]{scrartcl}	
\usepackage[T1]{fontenc}

\usepackage[english]{babel}															
\usepackage[protrusion=true,expansion=true]{microtype}				
\usepackage{amsmath}										
\usepackage{amsfonts}
\usepackage{amsthm}
\usepackage{latexsym}

\usepackage[pdftex]{graphicx}														
\usepackage{url}
\usepackage{amssymb}
\usepackage{bbm}
\usepackage{MnSymbol}

\usepackage[pdftex]{color}
\usepackage{eucal}
\usepackage{amscd}

\usepackage{sectsty}												
\allsectionsfont{\centering \normalfont\scshape}	

\usepackage{mathtools}

\usepackage{fancyhdr}
\newtheorem{theorem}{Theorem}

\newtheorem{proposition}{Proposition}
\newtheorem{corollary}{Corollary}
\newtheorem{definition}{Definition}

\numberwithin{equation}{section}		 
\numberwithin{proposition}{section}			 
\numberwithin{table}{section}				 
\numberwithin{definition}{section}
\numberwithin{theorem}{section}
\numberwithin{corollary}{section}
\numberwithin{exercise}{section}

\title{
		\vspace{-1in} 	
		\usefont{OT1}{bch}{b}{n}
		\normalfont \normalsize \textsc{} \\ [25pt]
		\huge  Differential Geometry  of Rigid Bodies Collisions and Non-standard Billiards
}

\author{\normalfont \large 
   C. Cox\footnote{Department of Mathematics, Washington University, Campus Box 1146, St. Louis, MO 63130}, 
\ R. Feres\footnotemark[1], 
\  W. Ward\footnotemark[1]
}

\begin{document}

\maketitle

\begin{abstract}
\begin{center} Abstract \end{center}
{\small The configuration manifold $M$ of a  mechanical system consisting of two unconstrained rigid bodies in $\mathbb{R}^n$, $n\geq 1$, 
is a manifold with boundary (typically with  singularities.) A complete description
of the system requires boundary conditions that specify how orbits should be continued after collisions.
A boundary condition  is the assignment of a {\em collision map} at each tangent space
on the boundary of $M$ that  gives the post-collision state of the system as a function of the pre-collision state. 
Our main result is a complete description of the space of linear collision maps satisfying energy and (linear and angular) momentum conservation, time reversibility,
and the natural requirement that impulse forces only act at the point of contact of the colliding bodies. 
These assumptions can be stated in   geometric language by making explicit a family of vector subbundles
of the tangent bundle to the boundary of $M$: the {\em diagonal,    non-slipping,} and {\em impulse} subbundles.
Collision maps at a boundary configuration are shown to be the isometric involutions that restrict to the identity on the non-slipping subspace.
The space of such maps  is naturally identified with 
the  union of Grassmannians of $k$-dimensional subspaces of $\mathbb{R}^{n-1}$, $0\leq k\leq n-1$, each subspace specifying the directions of {\em contact roughness}.
We then consider  {\em non-standard billiard systems},   defined by fixing the position of one of the bodies and 
allowing  boundary conditions
different from  specular reflection. We also make a few observations of a dynamical nature for  simple examples of
non-standard billiards and provide a sufficient condition for the billiard map on the space of boundary states to preserve the canonical (Liouville) measure on
constant energy
hypersurfaces. 
}
\end{abstract}

 \section{Introduction}
 The classical  theory of collisions of rigid bodies provides a very natural setting
 in which to explore the geometry and dynamics of mechanical systems on configuration
 manifolds with boundary. From this geometric perspective, the
response of the system to collisions between its   rigid moving parts
 is specified   by assigning  appropriate   boundary conditions
 that tell  how a trajectory should be continued once it reaches the boundary. 
  For example,
  in the  theory of billiard  dynamical systems, a topic that may be defined very broadly as
  the   study of Hamiltonian (more typically, geodesic flow) systems on 
   Riemannian manifolds with boundary, 
   one typically assumes that   trajectories reflect off the boundary specularly\----the
   simplest form of impact response compatible with the basic laws of mechanics such
   as energy conservation and time reversibility.  Billiard systems with more general boundary conditions
   have to our knowledge been investigated  very rarely.  One pertinent example is \cite{gutkin}, which 
   is restricted to $2$-dimensional billiards.
   (There is, of
 course, an extensive literature in engineering and applied physics about less idealized systems
 governed by impact interactions, but this literature is not concerned with  the  differential geometric  issues that are the main focus here.)

 Our first goal  is to classify boundary conditions for   systems defined by
 two unconstrained rigid bodies in $\mathbb{R}^n$, $n\geq 1$, under standard physical assumptions of energy conservation, 
linear and angular  momentum conservation, time reversibility, linearity of response, and another condition to be defined shortly that
 extends  momentum conservation and is 
 typically made implicitly in
 textbooks.  Collisions satisfying all of these properties will be called {\em strict}. 
  They are formally represented by linear maps $\mathcal{C}_q:T_qM\rightarrow T_qM$, where
  $M$ is the   configuration manifold equipped with the kinetic energy Riemannian metric, $q$ is a boundary configuration, and
  the tangent space $T_qM$ is the space of (pre- and post-) collision states. A boundary condition
  for the system then  consists of the (differentiable, measurable, random, etc.) assignment  
  of a strict collision map $\mathcal{C}_q$ to each boundary configuration $q\in \partial M$. 
 
 In dimensions greater than $1$, the collision map is not uniquely determined by the conditions of strict collision. It is well-known that
 the nature of the contact between the colliding rigid bodies  also needs to be specified. The standard case
 in which $\mathcal{C}_q$ is specular reflection corresponds to bodies  having     physically smooth surfaces.

Towards this classification   we
   identify a family  of subbundles of $T(\partial M)$
  arising naturally under the assumed physical laws and discuss some relationships among them.
  We then show   examples of trajectories of  non-standard (i.e., non-specular) boundary conditions 
and make a few observations about their dynamics based on numerical simulation. Although we   leave for  future work 
  a more 
systematic analysis of rough billiard dynamics, 
 we  give here sufficient conditions for the non-standard billiard system to leave invariant the natural volume measure on
a   constant energy manifold (derived from the canonical symplectic form). 
 The invariance of this {\em  billiard measure} makes it  
possible to bring  the  tools of ergodic theory (see \cite{chernov,serge}) to the  study of   non-standard billiard systems, 
 although we do not pursue this direction here. 

 \section{Statements of  the main results and examples}
 \subsection{Notation, terminology, and standing assumptions}
 For the most part we consider the unconstrained motion of two rigid bodies, represented by  the sets $B_1, B_2\subset \mathbb{R}^n$.
 We call these sets the {\em bodies in reference configuration}. Let $G=SE(n)$ denote the Euclidean group of orientation
 preserving isometries  of $\mathbb{R}^n$ equipped with the standard inner product. 
The bodies  are  assumed to be connected  $n$-dimensional submanifolds of $\mathbb{R}^n$ with  smooth boundaries.      
An {\em interior configuration} of the system consists of a pair $(g_1, g_2)\in G\times G$ such that $g_1(B_1)$ and $g_2(B_2)$
are disjoint sets. The  closure of the set of interior configurations, denoted $M$, will be called the {\em configuration space} of the
system, and its boundary is the set of {\em contact} (or {\em collision}) configurations.   $M$ has  dimension $2\, \text{dim}(G)$ and 
the nature of the boundary $\partial M$ will depend on geometric assumptions about the  
$B_j$. We will soon state a sufficient, and fairly general for our needs, condition  on the $\partial B_j$ for $M$ to be a submanifold
of $G\times G$ with smooth boundary.

The following notations are  used fairly consistently throughout the paper. Points in $B_j$ are  denoted $b, b_j$. Elements of the
Euclidean group, which is
  the semidirect product $G=SO(n)\ltimes \mathbb{R}^n$ of the  groups of rotations and translations, are written as
pairs $(A,a)$, where $A\in SO(n)$ and $a\in \mathbb{R}^n$. 
Elements of the Lie algebra $\mathfrak{g}=\mathfrak{se}(n)$
are written $(Z,z)$, possibly with   subscripts or superscripts, where $Z\in \mathfrak{so}(n)$ and $z\in \mathbb{R}^n$. 
The outward-pointing unit normal vector to the boundary of $B_j$ at  $b\in \partial B_j$
is denoted  $\nu_j(b)$. It is convenient to consider orthonormal frames $\sigma$ at $b\in \partial B_j$ {\em adapted} 
to the bodies, in the following sense: $\sigma:\mathbb{R}^n \rightarrow T_b\mathbb{R}^n\cong \mathbb{R}^n$ is
an element of $SO(n)$ such that $\sigma e_n=-(-1)^j \nu_j(b)$, where $e_n=(0, \dots, 0, 1)^\dagger$ is the 
last element of the standard basis $\{e_1, \dots, e_n\}$ of $\mathbb{R}^n$ and `$\dagger$'  indicates matrix
transpose.  

  \vspace{0.1in}
   \begin{figure}[htbp]
\begin{center}
\includegraphics[width=4.5in]{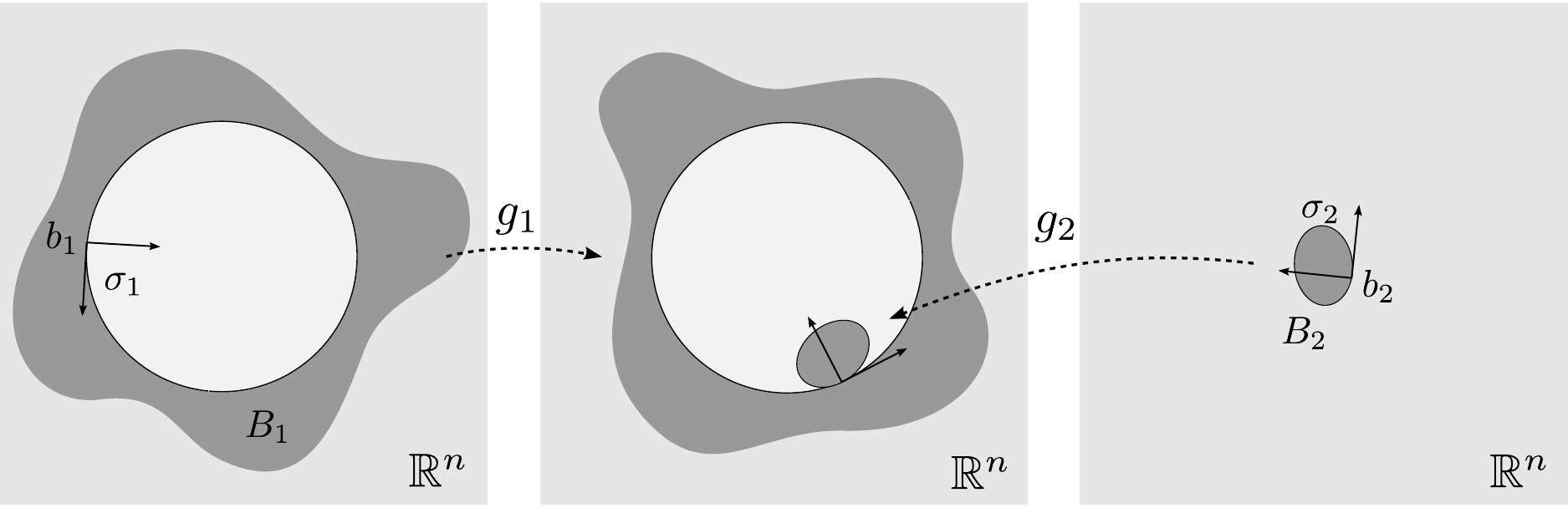}\ \ 
\caption{{\small   On the left and right are the bodies $B_1$ and $B_2$ in their
reference configuration in $\mathbb{R}^n$.  A configuration of the system of rigid bodies is given by a pair $(g_1, g_2)$ of
elements of the Euclidean group $G$.  A boundary configuration can be parametrized by
the tuple $(b_1, \sigma_1, b_2, \sigma_2)$ where    $b_j\in \partial B_j$ such that $g_1(b_1)=g_2(b_2)$
and   $\sigma_j$ is an orthonormal frame at $b_j$ as will be  explained in the text.}}
\label{parameters}
\end{center}
\end{figure}

Because $M$ is a submanifold of $G\times G$, each tangent space can be canonically identified with $T_{(e,e)}(G\times G)\cong\mathfrak{g}\oplus \mathfrak{g}$
by left-translation. Thus {\em states} of the system, defined as elements of $TM$, may be canonically identified with
tuples $(A_1, a_1, A_2, a_2, Z_1, z_1, Z_2, z_2)$. 
We sometimes indicate the state by $(q, \xi)$, where $q=(g_1, g_2)$ and $\xi=(Z_1, z_1, Z_2, z_2)\in \mathfrak{g}\times \mathfrak{g}$.
The position of  material point $b\in B_j$
in the given state is then $g_j(b)=A_jb+a_j$ and its velocity is $V(b)=A_j(Z_jb+z_j)$. The boundary configuration $(g_1, g_2)$ can also be parametrized,
up to an overall rigid motion of the two bodies  keeping  their   positions relative to each other unchanged, by 
$(b_1, \sigma_1, b_2, \sigma_2)$, where $b_j$ is in the boundary of $B_j$ and $\sigma_j$ is an adapted frame such
that $g_1(b_1)=g_2(b_2)$ and $A_1\sigma_1=A_2\sigma_2$.   The
tuple
$(b_1, \sigma_1 h, b_2, \sigma_2 h)$ corresponds to the same contact configuration, 
for all $h\in H$, where
$$H:=SO(n-1)=\{ A\in SO(n):Ae_n=e_n\}.$$
These notions are illustrated in Figure \ref{parameters}.

       Let $\Pi_b$ denote the
    orthogonal projection from $\mathbb{R}^n$ to the tangent space to the boundary of $B_j$ at a boundary point $b$.
The orthogonal projection to $\mathbb{R}^{n-1}=e_n^\perp$ will be denoted $\Pi$. 
   The {\em shape operator} of the boundary of $B_j$ at the point $b$ is the
   linear map defined by
   $$S_b:\nu_j(b)^\perp\rightarrow \nu_j(b)^\perp, \ \ S_b v=- D_v \nu_j$$ where $D_v$ is
   directional derivative in $\mathbb{R}^n$. We say that   $S:\mathbb{R}^{n-1}\rightarrow \mathbb{R}^{n-1}$
   is the shape operator $S_b$ in  the adapted frame $\sigma$ at $b$  if $\sigma S =S_b\sigma$.
   The notation $\text{Ad}_\sigma S= \sigma S \sigma^{-1}$ will be used often to indicate conjugation.
   
   So as not to get distracted by regularity issues, we assume that 
  the configuration manifold $M$ has smooth boundary and that  each boundary configuration 
  corresponds to the bodies being in contact at a single common point. Proposition \ref{regularity},
  which is a special case of Proposition \ref{regular},
  gives a sufficient condition for $M$ to be nice in this respect. The hypotheses of
  the proposition will be assumed to hold throughout the paper.
  
  \begin{proposition}\label{regularity}
  Suppose that $S_1+S_2$ is nonsingular for every  relative configuration $(b_1, \sigma_1, b_2, \sigma_2)$ of the rigid bodies, 
 where   $S_j:\mathbb{R}^{n-1}\rightarrow \mathbb{R}^{n-1}$
is the shape
  operator of the boundary of $B_j$ at $b_j$ in  the adapted frame $\sigma_j$.
  Then $M$ is a smooth manifold of dimension $2\, \text{\em dim} G$ with smooth boundary, and each boundary point
  $q=(g_1,g_2)$
  represents a  configuration  with a unique point of contact. Moreover the map that associates to  
  $q\in \partial M$ the uniquely determined pair $(b_1, b_2)\in \partial B_1\times \partial B_2$  such that $g_1(b_1)=g_2(b_2)$
  is smooth.
  \end{proposition} 
   
  We call the   $b_1, b_2$    associated to   $q\in \partial M$ under the condition of Proposition
  \ref{regularity} the {\em contact points} (in the reference configuration) associated to boundary point $q$.

\subsection{The kinematic bundles}
   If $a, b\in \mathbb{R}^n$, let $a\wedge b\in \mathfrak{so}(n)$  be  the $n$-by-$n$ matrix such that 
  $(a\wedge b)_{ij}=a_j b_i - a_i b_j$.     If $a, b$ are orthogonal unit vectors, $a\wedge b$ is the infinitesimal generator
  of the one-parameter group in $SO(n)$ that rotates  the plane spanned by $a$ and $b$ and  fixes pointwise the  orthogonal complement of
  that plane. The  boundary state of the two-body system consists of the boundary configuration
  $q=(g_1, g_2)\in G\times G$ and velocities $\xi=(Z_1, z_1, Z_2, z_1)\in \mathfrak{g}\times \mathfrak{g}$.

   We now define the {\em kinematic bundles}.
   Given   $q=(g_1, g_2)\in \partial M$, with 
   $g_j= (A_j, a_j)$ and associated  contact points $b_1, b_2$, 
consider  the following linear relations on the  $\xi\in T_q(\partial M)$, where $N_j=\partial B_j$ and  $\nu_j=\nu_j(b_j)$:
    \begin{enumerate}
\item[$R_1:$] $\nu_1\cdot (Z_1b_1+z_1)=\nu_2\cdot(Z_2 b_2+z_2)$
\item[$R_2:$]  $A_1 (Z_1b_1+z_1)=A_2 (Z_2 b_2+z_2)$
\item[$R_3:$] $\text{Ad}_{A_j}Z_j=W+\nu_j\wedge w_j$ for $W\in \mathfrak{so}(n)$ and $w_j\in T_{b_j}N_j$, $j=1,2$ 
\item[$R_4:$] $\text{Ad}_{A_1}Z_1=\text{Ad}_{A_2}Z_2$.
 \end{enumerate}
As already noted, $A_j(Z_jb_j+z_j)$ is the velocity of the contact point $b_j$ in the given state.
Observe that relation $R_2$ implies $R_1$. 
 It will be shown later that  $$T_q(\partial M)\cong \{\xi\in \mathfrak{g}\times\mathfrak{g}:R_1\}. $$
 The physical interpretation of the kinematic bundles is as follows. A state satisfying   $R_1$ has the property that the contact points
 have zero relative velocity  in the normal direction to the plane of contact. Relation $R_2$ is satisfied
exactly when the contact points are not moving at all relative to each other at the moment of contact.
This means that the contact points do not slip past each other. Relation $R_3$  describes a state
in which the tangent spaces to the bodies at the point of contact do not experience a relative rotation (on
that tangent space). Thus it is a condition of {\em non-twisting}. Together $R_2$ and $R_3$
describe a state in which the bodies are rolling  on each other.
 
 \begin{definition}[Kinematic bundles]\label{kinematic bundles}
 Let $\mathfrak{S}$, $\mathfrak{R}$, and $\mathfrak{S}$ be the vector subbundles of $T(\partial M)$ defined by
 \begin{align*}
\mathfrak{S}_q &\cong \{\xi\in  \mathfrak{g}\times \mathfrak{g}:  R_2\}\\
\mathfrak{R}_q &\cong \{\xi\in  \mathfrak{g}\times \mathfrak{g}:  R_2, R_3\}\\
\mathfrak{D}_q&\cong \{\xi\in  \mathfrak{g}\times \mathfrak{g}:  R_2, R_3, R_4\}.
\end{align*}
Note that $\mathfrak{D}\subset \mathfrak{R}\subset \mathfrak{S}$. We refer to $\mathfrak{S}$ as the {\em non-slipping} subbundle,
$\mathfrak{R}$ the {\em rolling} subbundle, and $\mathfrak{D}$ the {\em diagonal} subbundle. 
\end{definition}

It will be shown that the diagonal subbundle $\mathfrak{D}$ is the tangent bundle to the orbits of the  action of $G$ on $M$  by
left translations: $g(g_1, g_2):=(gg_1, gg_2)$. 
Later in   we give a different definition of these subbundles, Definition \ref{bundles}, that makes their physical interpretation more clear.
Then what is stated above as a definition is derived in Section \ref{nsrd}.

\subsection{The kinetic energy   metric and the impulse subbundle}\label{definitions mechanics}
Suppose now that the bodies $B_j$ are assigned mass distributions represented by
 finite positive measures $\mu_j$ supported on $B_j$. Let $m_j:=\mu_j(B_j)$ be
the {\em mass} of   $B_j$. We may assume without loss of generality that
$\mu_j$ has zero first moment:
 $\int_{B_j} b\, d\mu_j(b) =0.$ This is to say that $B_j$ has {\em center of mass} at the origin of $\mathbb{R}^n$. 
 The matrix of  second moments  of $\mu_j$ is 
  $L_j=(l_{rs})$, with entries  
 $$l_{rs}=\frac1{m_j} \int_{B_j} b_r b_s \, d\mu_j(b). $$
 We call $L_j$ the {\em inertia matrix} of body $B_j$.
 This matrix  induces a  map $\mathcal{L}_j$  on $\frak{so}(n)$ that associates to
  $Z\in \mathfrak{so}(n)$ the  matrix
 $\mathcal{L}_j(Z)=L_jZ+ZL_j\in \mathfrak{so}(n)$.

 \begin{definition}[Kinetic energy Riemannian metric]\label{Riemannianmetric}
 Given $q\in M$ and $u,v\in T_q M$, define the symmetric non-negative form on $T_qM$ by
$$\langle u,v \rangle_q=\sum_{j} m_j\left[ \frac12 \text{\em Tr}\left(\mathcal{L}_j(Z^u_j){Z^v_j}^\dagger\right)+ z^u_j\cdot z^v_j\right]$$
where $(Z^u_1, z_1^u,Z_2^u, z_2^u)$ and $(Z^v_1, z_1^v,Z_2^v, z_2^v)$  are the translates to 
  $\mathfrak{g}\times \mathfrak{g}$
of   $u, v$. 
 When the above bilinear form is positive definite  we call it the {\em kinetic energy Riemannian metric} on $M$.
Denoting by $\|\cdot\|_q$ the corresponding norm at $q$, we call $\frac12\|v\|^2_q$ the {\em kinetic energy} associate to  state $(q, v)$.
 \end{definition}
 
 The kinetic energy function given in Definition \ref{Riemannianmetric}
 is easily shown (as indicated later) to come from  integration with respect to 
 the mass distribution measures of (one-half of) the Euclidean square norm of  
 the velocity of material point $b$ over the disjoint union of $B_1$ and $B_2$. Thus Definition \ref{Riemannianmetric}
 agrees with the  standard textbook definition of kinetic energy.  
 It is also clear  that the metric   is invariant under the left-action of $G$ on $M$. 
 Note that the boundary of $M$ is a $G$-invariant set.

  For each  $u\in \mathfrak{g}$ we define vector field $q\mapsto \tilde{u}_q\in T_qM$, $q=(g_1, g_2)$,  by
  $$ \tilde{u}_q :=  \left. \frac{d}{dt}\right|_{t=0} e^{t u}q.$$
We call $u\mapsto \tilde{u}$ the {\em infinitesimal action} derived  from the left $G$-action on $M$ and $\tilde{u}$
the {\em vector field associated} to $u\in \mathfrak{g}$.

\begin{definition}[Momentum map]
 The  map $\mathcal{P}^\mathfrak{g}:TM\rightarrow \mathfrak{g}^*$ defined by
  $$\mathcal{P}^\mathfrak{g}(q, \dot{q})(u) =\langle \dot{q}, \tilde{u}_q \rangle_q$$
  is called the {\em momentum map} associated to the $G$-action on $M$. 
\end{definition}

The most straightforward way of introducing dynamics into the system is through Newton's second law. There
are several equivalent forms of it as we note later. The following  is particularly convenient for our
needs.
 We first define a {\em force field} (possibly time dependent) as a bundle map $F:TM\rightarrow T^*M$.
 Given a state $(q, \dot{q})$, $q=(g_1, g_2)$, each component $F_j$ of $F$ can be pulled-back 
 to $\mathfrak{g}^*$ using right-translation $R_{g_j}$, so it makes sense to write
\begin{equation}\label{Newton}\frac{d}{dt}\mathcal{P}_j^\mathfrak{g}(q,\dot{q}) = R^*_{g_j} F_j.\end{equation}
This  is Newton's second law written as a differential equation on  the co-Lie algebra   of $G$. Other useful
forms are mentioned  later. One of them is indicated in the next proposition, in which we use the notation
$F^\#$ for the dual of $F$ with respect to the left-invariant Riemannian metric and write
$$(Y_j(t, q,\dot{q}),y_j(t,q,\dot{q})):= (dL_{g_j})_e^{-1} F^\#(t,q,\dot{q}) \in \mathfrak{g}.$$
Here we are using the differential of the left-translation map $L_{g_j}$.
\begin{proposition}\label{newton2}
The equation $\frac{d}{dt}\mathcal{P}_j^\mathfrak{g}(q,\dot{q}) = R^*_{g_j} F_j$ is equivalent to
\begin{align*}
m_j\left(\mathcal{L}_j\dot{Z}_j- [\mathcal{L}_jZ_j, Z_j]\right)&=\mathcal{L}_j(Y_j)\\
m_j \dot{v}_c&= A_jy_j
\end{align*}
where $g_j=(A_j, a_j)$ and $v_c=A_j z_j$ is the velocity of the center of mass of body $B_j$. 
\end{proposition}

We  assume that  $F$ results from the integrated effect of  forces acting on the
individual material points. That is, we assume that there exists a $\mathbb{R}^n$-valued
measure $\varphi_j$ on $B_j$ parametrized by $TM$ from which $F$ is obtained by integration:
$$F(q,v)(u)= \int_{B_j} V_u(b)\cdot \, d\varphi_{j, q,v}(b)$$
for all $u\in T_qM$, where $V_u(b)$ is the velocity of the material point $b$ in the state $(q,u)$.  
Of special interest for us are the forces  involved in the collision process.
These  {\em impulsive} forces are characterized by being  very intense and of   very short duration, applied  on a single point\----the point of
contact in each body. 

That the forces act on each body only at the point of contact greatly restricts the
right-hand side of the equation of motion in Proposition \ref{newton2}. This  is  indicated
in the next proposition.

\begin{proposition}\label{proposition two equations}
We suppose that the force field $F_j$ acting on body $B_j$ is such that the force distribution measure $\varphi$
is singular, concentrated at the point $b_j$. Then the equations of motion of Proposition \ref{newton2}
reduce to
\begin{align*}
m_j\left(\mathcal{L}_j\dot{Z}_j- [\mathcal{L}_jZ_j, Z_j]\right)&=b_j \wedge y_j\\
m_j \dot{v}_c&= A_jy_j
\end{align*}
\end{proposition}

For ideal  impulsive forces (of infinite intensity and infinitesimal duration),  momentum should change
discontinuously. Integrating Equation \ref{Newton} over a very short time interval $[t^-,t^+]$ around $t$ 
produces a nearly discontinuous change in momentum while keeping the configuration essentially
unchanged. We have informally
 $$\mathcal{P}_j^\mathfrak{g}(q,\dot{q}_+)-\mathcal{P}_j^\mathfrak{g}(q,\dot{q}_-)=\int_{t_-}^{t_+} R^*_{g_j}F_j\, ds= \text{ Impulse at $t$}. $$
It is not necessary for our needs to  make more precise the limit process suggested by this expression. From it we
obtain the form of the change in momentum after impact, which is given in the next proposition. 
Let $q=(g_1, g_2)\in \partial M$ be a collision configuration and 
  denote by $$(Z_1^\pm, z_1^\pm, Z^\pm_2, z^\pm_2)\in T_qM$$
  the post- ($+$) and pre- ($-$) collision velocities of the two rigid bodies.
   \begin{proposition}[Velocity change due to impulse at contact point]\label{velocitychange}
   Given pre-collision velocity $(Z_1^-, z_1^-, Z^-_2, z^-_2)$ there exist $u_1, u_2\in \mathbb{R}^n$ such that
    \begin{align*}
 z_j^+  &=  z_j^-+u_j\\
 Z_j^+ &=  Z_j^- + \mathcal{L}_j^{-1}( b_j  \wedge u_j ).
 \end{align*}
 Under conservation of linear momentum $m_1 A_1 u_1+m_2 A_2 u_2=0$ holds. 
   \end{proposition}
   
   The proof of the above proposition is given in Section \ref{several}.
   The assumption that
 impulsive forces of one body on the other at the moment of impact are applied at the point of contact
 is a strong constraint. One can in principle  conceive of force fields of relatively long range,  acting throughout the bodies, that are briefly switched on at the moment of
impact, then switched off as soon as the bodies lose contact. 
More realistically, the bodies could suffer a deformation around the region of impact, creating a small neighborhood of
contact. Of course this goes beyond the rigid body model. 
 Here it is  assumed  that these possibilities do not happen,
 and that any effect of one body on the other can only be transmitted
through  the single point of contact between them.   
   
   If   $L_j$ is non-negative definite of rank at least $n-1$,   $\mathcal{L}_j$
is invertible. With this in mind, 
   Proposition \ref{velocitychange} suggests the following definition.
   
   \begin{definition}[Impulse subbundle]
     The {\em impulse subbundle} of $TM$ (over the base manifold $\partial M$) is defined so that its fiber at $q\in \partial M$ is  the subspace
  $$\mathfrak{C}_{{q}}=\left\{((\mathcal{L}_1^{-1}(b_1\wedge u_1), u_1),(\mathcal{L}_2^{-1}(b_2\wedge u_2), u_2)): u_j\in \mathbb{R}^n,
  m_1 A_1 u_1+m_2 A_2 u_2=0\right\}. $$
   \end{definition}

We  have now the following  vector subbundles of $i^*(TM)$, where $i:\partial M\rightarrow M$ is the inclusion map:
$ \mathfrak{D}\subset\mathfrak{R}\subset \mathfrak{S}\subset T(\partial M) \text{ and } \mathfrak{C}.$ The latter subbundle is the
only one  that depends on the mass distributions.  
 \begin{theorem}\label{orthogonal}
 The impulse subspace  $\mathfrak{C}_{{q}}$   is the orthogonal complement of the 
 non-slipping subspace $\mathfrak{S}_{{q}}$ and contains the unit normal vector $\mathbbm{n}_q$. Therefore,
 $$ T_qM=\mathfrak{S}_q \oplus \left(\mathfrak{C}_q\ominus \mathbb{R}\mathbbm{n}_q\right)\oplus \mathbb{R}\mathbbm{n}_q$$
 is an orthogonal direct sum.
  \end{theorem}
A physical  interpretation of this orthogonal decomposition will be given  shortly.

\subsection{Collision maps}
 Let
$\mathbbm{n}_q$ be  the unit normal vector to $\partial M$ pointing into $M$ at a boundary configuration $q$.
Define the half-spaces
$$T^+_qM:=\{v\in T_qM:\langle v, \mathbbm{n}_q\rangle \geq 0\}=-T^-_qM.$$
We call any $\mathcal{C}_q:T^-_qM\rightarrow T^+_qM$   a
{\em collision map} at $q$. 
By a {\em boundary condition} we mean the assignment of such a map $\mathcal{C}_q$ to each $q\in \partial M$. 
 We only consider here {\em linear} collision maps; that is,  $\mathcal{C}_q$    extends to a linear map on $T_qM$. 
 
 \begin{definition}[Strict collision maps]\label{stcoll}
 A  collision map $\mathcal{C}_q$ at   $q\in \partial M$ is {\em strict} if the following 
  hold for all $u, v\in T_qM$:
   \begin{enumerate}
 \item {\em Conservation of energy}: 
  $ \langle \mathcal{C}_qv,\mathcal{C}_q u\rangle_q = \langle v,u\rangle_q$. That is, 
 $\mathcal{C}_q$ is a linear isometry. 
 \item {\em Conservation of momentum}: $\mathcal{P}^\mathfrak{g}(q,\mathcal{C}_q v)=\mathcal{P}^\mathfrak{g}(q,v).$  
 \item {\em Time reversibility}: $\mathcal{C}_q^2 =\text{\em Id}$  (a linear involution).
 \item {\em Impulse at the point of contact}:   $ \mathcal{C}_q v-v\in \mathfrak{C}_q$.
 \end{enumerate}
\end{definition}
      
 \begin{proposition}\label{proposition} Condition 2 of Definition \ref{stcoll} is equivalent to assuming that 
 $\mathcal{C}_q$ restricts to the identity map on $\mathfrak{D}_q$. Condition 4
  is equivalent to  
  $\mathcal{C}_q\mathfrak{C}_q=\mathfrak{C}_q$ and $(\mathcal{C}_q-\text{\em Id})\mathfrak{C}_q^\perp=0$.
    \end{proposition}  
    Thus   energy conservation and impulse at a single contact point   are together equivalent
 to $\mathcal{C}_q$ being   the identity on   $\mathfrak{S}_q$.
 In this sense, condition (4) of Definition \ref{stcoll}  can be regarded as generalizing
 momentum conservation as we note in Proposition \ref{proposition}. In fact, 
 conservation of momentum amounts to $\mathcal{C}_q$ being the identity on $\mathcal{D}_q$, 
 whereas 4 and 
 Theorem \ref{orthogonal} imply that $\mathcal{C}_q$ is the identity on the bigger
 subspace $\mathfrak{S}_q$. An intermediate condition is that $\mathcal{C}_q$ restricts
 to the identity  on the rolling subspace $\mathfrak{R}_q$.
      
       \begin{corollary}  
Strict collision maps   are the  linear isometric involutions of $T_qM$, $q\in \partial M$,
that restrict to the identity map on the non-slipping   subspace $\mathfrak{S}_q$.
 \end{corollary}

  Collision maps have eigenvalues $\pm 1$. The map $P_\pm:=\frac{I\pm \mathcal{C}_q}{2}$ is the orthogonal projection
  to the eigenspace associated to eigenvalue $\pm 1$.  
   \begin{definition}
   The dimension of the eigenspace of $\mathcal{C}_q$ associated to eigenvalue $-1$ will be called the  {\em roughness rank} of  $\mathcal{C}_q$.
   The image of the orthogonal projection $P_-$ (a subspace of $\mathfrak{C}_q$) will be called the {\em roughness subspace} at $q$. 
   \end{definition}
   
   The unit normal vector $\mathbbm{n}_q$ is always contained in the impulse subspace $\mathfrak{C}_q$ and it must necessarily
   be in the $-1$-eigenspace of $\mathcal{C}_q$. 
   
 \begin{corollary}\label{corollary} Let $n$ be the dimension of the ambient Euclidean space. Identifying $$\mathfrak{C}_q\ominus \mathbb{R}\mathbbm{n}_q\cong\mathbb{R}^{n-1},$$  
 the set  of strict  collision maps is the set of  $C\in O(n-1)$ such that $C^2=I$.
  Writing 
  $$\mathcal{J}_k:= O(n-1)/(O(n-k-1)\times O(k)),$$
  then the set of strict  collision maps at  any given boundary point  is  $\mathcal{J}_0\cup \dots \cup \mathcal{J}_{n-1}.$
  Moreover, $\dim \mathcal{J}_k=k(n-k-1)$ and $k$ is the roughness rank at $q$. We call $\mathcal{J}_k$ the {\em Grassmannian
  of rough subspaces} having roughness  rank $k$. 
 \end{corollary}
     
     It is easy to compute the dimensions of the Grassmannians $\mathcal{J}_k$ for strict collision maps.  They are given, up to dimension $5$,
     by the following table:
         $$\text{dim } \mathcal{J}_k: \ \left[ \begin{array}{cc|ccccc} &   k & 0 & 1 & 2 & 3 & 4   \\ n  &  &  &  &  &  &   \\\hline 1 &  & 0 &  &  &  &     \\ 2  &  & 0 & 0 &  &  &    \\ 3 &   & 0 & 1& 0 &  &     \\ 4  &  & 0 & 2 & 2 & 0 &\\
  5 & & 0 &3 & 4 & 3 & 0
    \end{array}\right]$$
      The table shows that in dimension $1$ there is a unique strict collision map; in dimension $2$ there
      are exactly $2$ possibilities; and in dimension $3$ there is one possibility of roughness rank $0$ given
      by the standard reflection map, one possibility for maximal roughness rank $2$, and a one-dimensional
      set of possibilities for roughness rank $1$ parametrized by the  lines through the origin in $\mathbb{R}^2$. 
   For general $n$, the unique collision map of maximal roughness rank will be referred to as 
   the {\em completely rough} reflection map.

 \subsection{Non-standard billiard systems}  
 We have so far considered systems consisting of two unconstrained rigid bodies.
 The results of this paper can be extended to situations in which one body  or both  are subject to 
  holonomic and non-holonomic constraints. We will explore  this extension  more systematically elsewhere.
  Here we  consider only the   case in which body $B_1$ remains fixed in place 
  whereas $B_2$ is unconstrained except for the condition that it cannot overlap with $B_1$. 
The term {\em billiard system}   will refer to a  system of this kind where $B_2$ is 
   a ball with rotationally symmetric mass distribution.  The   system will be called   non-standard
   if the (strict) collision maps are not all specular reflection.

\vspace{0.1in}
   \begin{figure}[htbp]
\begin{center}
\includegraphics[width=2.5in]{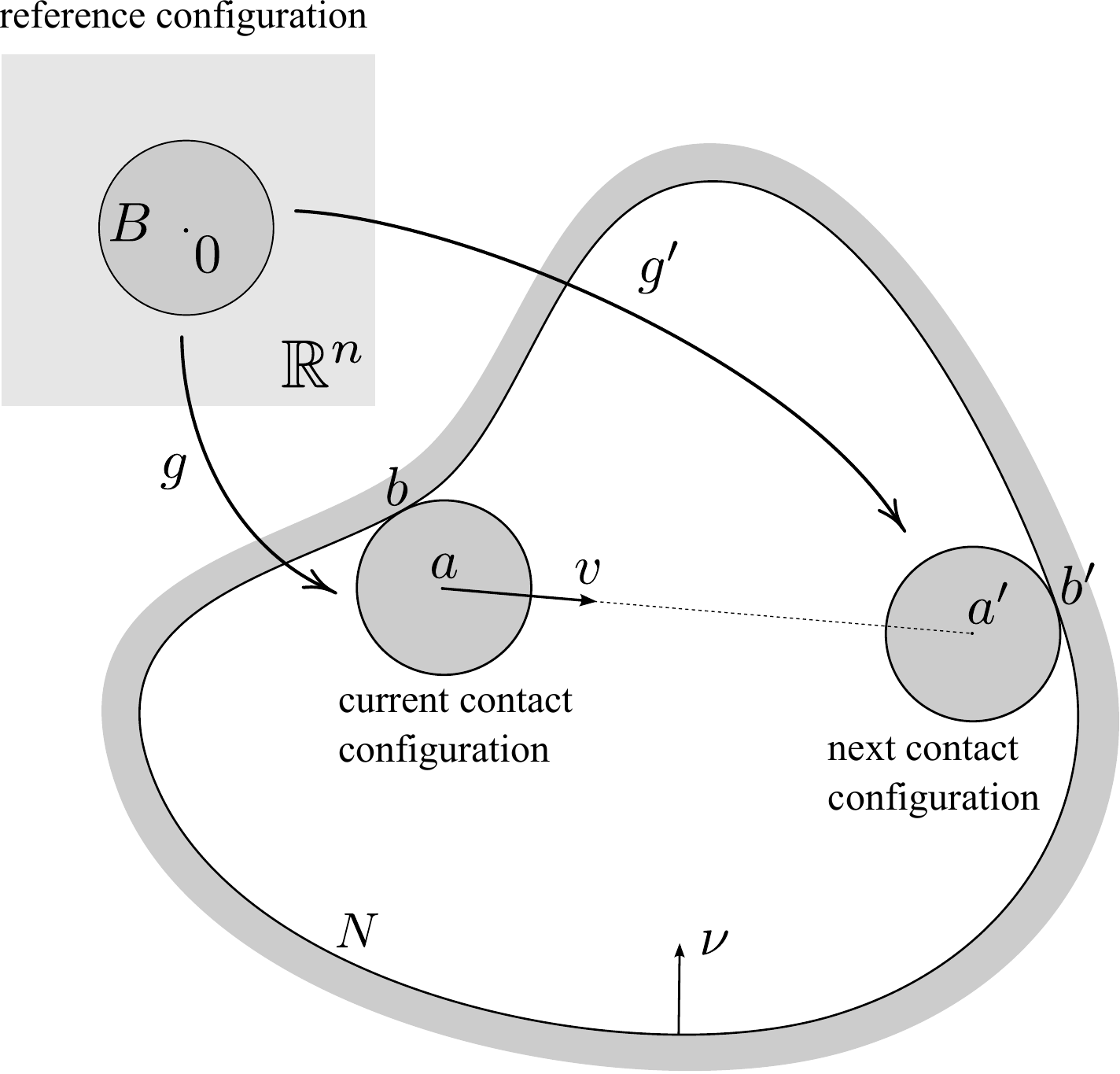}\ \ 
\caption{{\small   A billiard system. Body $B_1$ is kept fixed (the billiard table) and $B:=B_2$
is a ball of rotationally symmetric mass distribution that can move freely in the complement of $B_1$. 
Given the contact point $b\in N$ and the post-collision velocity $v$ of the center of mass of  $B$,
the point of contact of the next collision will be written  $b'=T(b,v)$.}}
\label{nonstandard}
\end{center}
\end{figure}

Let $R$ denote the radius of $B:=B_2$ and $m$ its mass. 
Due to rotational symmetry of the mass distribution of $B$,  the matrix of inertia $L$ is scalar, that is, $L=\lambda I$.
For example, a simple integral calculation shows that if $B$ has uniform mass distribution, then
$\lambda=\frac{R^2}{n+2}.$ (Recall that we have defined $L$ as the matrix of second moments of the mass distribution measure
divided by the total mass, so   $m$ does not appear in $\lambda$.) The (smooth) boundary of $B_1$ will be denoted $N$ and
the unit normal vector field on $N$ pointing into the region of free motion of $B$ will be denoted $\nu$.
Trajectories of the billiard system   are sequences of states: $(g_0,\xi_0), (g_1, \xi_1), \dots,$
where
$$ (g_i, \xi_i)\in SE(n)\times \mathfrak{se}(n) \cong TSE(n), \ \  g_i=(A_i,a_i),\ \  \xi_i=(Z_j,z_j).$$
Here $g_i$ is the contact configuration  and $\xi_i$ the post-collision velocities 
in the body frame (reference configuration) at the $i$th collision.

To each contact state $(g, \xi)=(A, a, Z, z)$ is     associated  a unique contact point 
$b\in N$ and the post-collision  velocity $v=Az$   of the center of mass.
 The center of mass of $B$ in configuration $g$ is $a$ and the velocity of any given material point
$b\in B$ is $V(b):= A(Zb+z)$. When it is necessary to distinguish  points in $N$ and in $B$ we write $b\in N$ and
$b_\circ\in B$.  The unit normal vector to $B$ at $b_\circ$ will be written $\nu_\circ(b_\circ)=b_\circ/R.$
The point of contact at the next collision, which only depends on $b$ and $v$,  will be denoted $b'=T(b,v)$.
 See Figure
\ref{nonstandard}.

One step of the billiard motion, $(g,\xi)\mapsto (g',\xi')$, amounts to the following operations. 
\begin{enumerate}
\item From the current collision state $(g,\xi)$ at time $t$ one obtains the contact point $b\in N$
and  velocity $v$ of the center of mass $a$ of $B$ where $g=(A, a)=(A(t),a(t))$.
It should be kept in mind that $\xi=(Z,z)$ describe post-collision velocities so $v=Az$ points
into the region of free motion of the ball.
\item Obtain the  contact point $b'=T(b,n)\in N$  and the  time   $t'=t+\tau$ of the next collision.
\item Obtain the next pre-collision state: $(g',\xi^-)$ where $g'=(A',a')=(A(t+\tau), a(t+\tau))$, $\xi^-=(Z^-,z^-)$,
and $$ A' = Ae^{\tau Z},\  a'=a+\tau A z,\  Z^-=Z,\  z^-=e^{-\tau Z}z.$$
This is the free (geodesic) motion between collisions.  Observe that $a'=b'+R\nu(b')$.
\item  Let $b_\circ=(g')^{-1}b'\in \partial B$ be the contact point on the ball in the reference configuration at the next collision
and denote by $\Pi_\circ, \Pi_\circ^\perp$  the orthogonal projections  to the tangent space to $\partial B$ at $b_\circ$ 
and  to $\mathbb{R} \nu_\circ(b_\circ)$, respectively.  Note that $A'\nu_\circ(b_\circ)=-\nu(b')$.
\item Finally, compute $\xi'=(Z',z')$ from $(Z^-,z^-)$ using the choice of collision map.  
It will be shown that 
\begin{equation}\label{update} (Z',z')=\left(Z^-- \frac{\alpha}{2\lambda}b_\circ \wedge (I-\mathcal{T}) V^-, z^- -\alpha(I-\mathcal{T}) V^--2\Pi^\perp_\circ z^-\right),\end{equation}
where $\alpha:=1/(1+R^2/2\lambda)$, $V^-=\Pi_\circ (Z^-b_\circ +z^-)$, and $\mathcal{T}$ is  a linear involution
on $T_{b_\circ} (\partial B)$ corresponding to a choice of collision map. For specular reflection $\mathcal{T}=I$ and
for completely rough collisions $\mathcal{T}=-I$. 
\end{enumerate}

\subsection{Examples of non-standard billiards}
We assume in all examples  the uniform mass distribution on the ball $B$ so $\lambda=R^2/(n+2)$,
where $R$ is the radius of $B$. Let first 
  $n=2$. In this case the only non-standard collision map corresponds to $\mathcal{T}=-I$.
Elements of the rotation group  are parametrized by the angle of rotation $\theta$ and elements of the Lie algebra of $SO(2)$
are  written as $\dot{\theta}J$,  where $J$ is the rotation matrix by $\pi/2$ in the counterclockwise direction.
Together with the standard coordinates $(x,y)$ we obtain  coordinates $(\theta, x, y)$ on  $SE(2)$.
It will be convenient to make the   coordinate change:
$x_0=R\theta/\sqrt{2},  x_1=x,  x_2=y$. This yields coordinates $(x_0, x_1, x_2, \dot{x}_0, \dot{x}_1, \dot{x}_2)$
on the billiard   state space.  
We also write $v_0=\dot{x}_0$ and  $v=(\dot{x}_1, \dot{x}_2)^\dagger$ for the velocity of the center of mass of
the disc. 

The choice of coordinates is made so that the kinetic energy Riemannian metric becomes, up to multiplicative constant,
the standard Euclidean metric. Then it can be derived from Equation \ref{update} that
the post-collision velocities $(v^+_0,v^+)$ after collision at point of contact $b\in N$ is  the function of 
  the pre-collision velocities $(v_0^-, v^-)$ given by
\begin{equation}\label{update2}
\begin{aligned}
v_0^+&= -\frac13 v_0^- + \frac{2\sqrt{2}}{3} v\cdot (J \nu(b))\\
v^+&=\left[\frac{2\sqrt{2}}{3} v_0^- + \frac13 v^-\cdot (J \nu(b))\right]  J\nu(b) - v^-\cdot \nu(b) \nu(b).
\end{aligned}
\end{equation}
Thus the state updating equations for a $2$-dimensional non-standard billiard system is as follows. If $\tau$ is
the time of free flight between the two consecutive collisions and  setting $\mathbf{x}=(x_0,x_1,x_2)$, 
$\mathbf{v}=(v_0,v_1,v_2)$, then the billiard map giving the next state $(\mathbf{x}',\mathbf{v}')$
as a function of the present state $(\mathbf{x}, \mathbf{v}^-)$ is
$(\mathbf{x}', \mathbf{v}')=(\mathbf{x}+\tau\mathbf{v},\mathbf{v}^+ ) $
where $\mathbf{v}^+$ is related to $\mathbf{v}^-$ according to Equations \ref{update2}.
The geometric interpretation of those equations   is explained in Figure \ref{interpretation}.
Note the role played by the
angle $\beta$ defined by $\cos\beta=1/3$, $\sin\beta=2\sqrt{2}/3$. In \cite{gutkin} it
is observed that $\beta$ is the dihedral angle  of a regular tetrahedron.
In the  
figures to follow we only  indicate  the position of the center of the disc; we
draw   a smaller  table whose boundary is at a distance $R$ from the boundary of the original table and we  imagine the
center of the ball  as a point mass bouncing off  the boundary of this smaller region.

\vspace{0.1in}
   \begin{figure}[htbp]
\begin{center}
\includegraphics[width=4.0in]{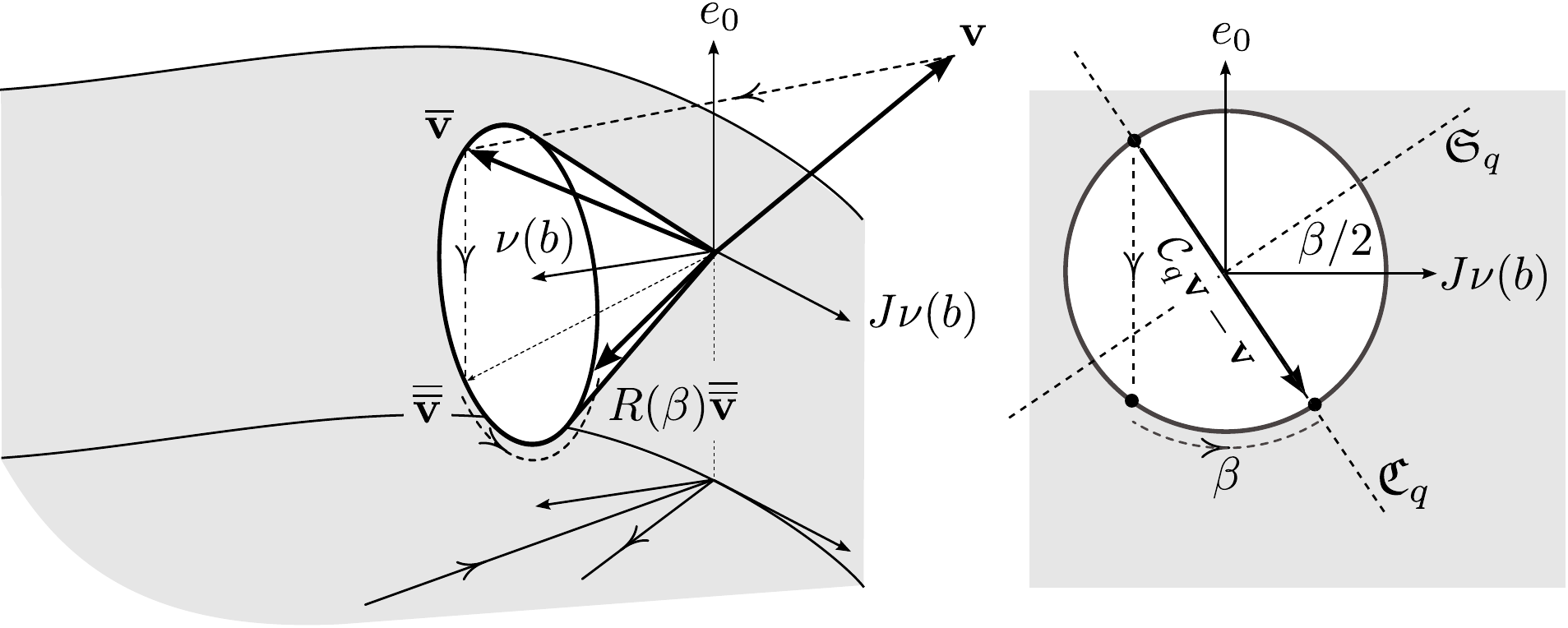}\ \ 
\caption{{\small  Geometric interpretation of the rough reflection in angle-position space for $n=2$. The $e_0$ component of the incoming velocity $\mathbf{v}$ is
the scaled angular velocity $v_0=R\dot{\theta}/\sqrt{2}$. The outgoing velocity is obtained by first reflecting $\mathbf{v}$ specularly
on the   plane spanned by $e_0$ and $J\nu(v)$ to find $\overline{\mathbf{v}}$, then
reflecting the latter specularly on the  plane spanned by $\nu(b)$ and $J\nu(b)$, and finally  rotating the resulting
vector by $\beta$ as indicated. As noted in \cite{gutkin}, $\beta$ is the
dihedral angle  of a regular tetrahedron. }}
\label{interpretation}
\end{center}
\end{figure}

Next we show examples of trajectories 
of systems with rough collisions. The examples are given here without much analysis.
We leave  the  more systematic study of the dynamics of such systems for another article.
As a first illustration, consider the case of a circular billiard table.
The typical trajectory is shown in Figure \ref{caustics} and a few more examples are shown in
  Figure \ref{circle}.

The following proposition captures the main properties of trajectories of circular rough billiards
readily observed in the Figure \ref{caustics}.

 \vspace{0.1in}
   \begin{figure}[htbp]
\begin{center}
\includegraphics[width=1.5in]{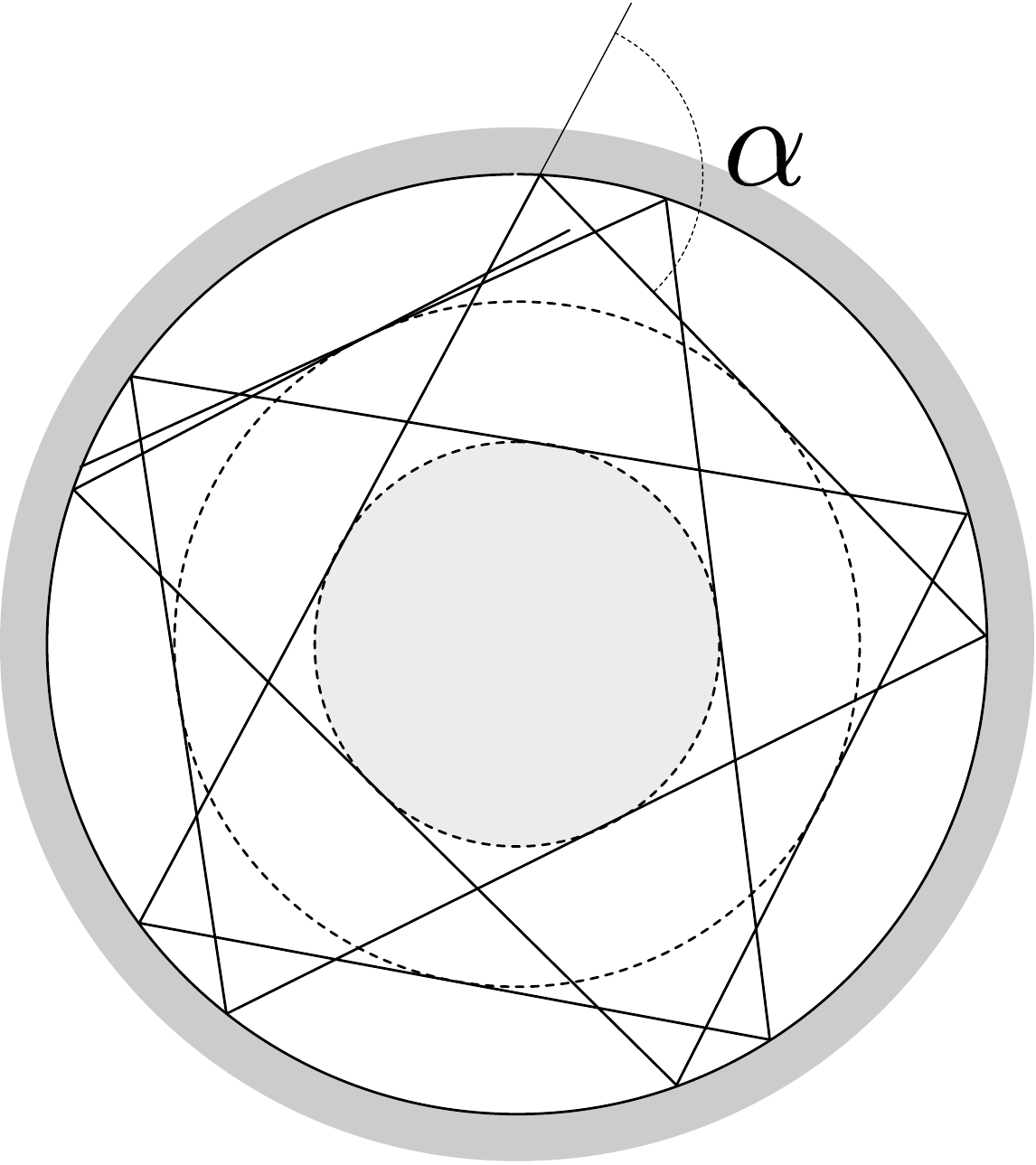}\ \ 
\caption{{\small   Caustics of a circle billiard with rough collisions consist of pairs of concentric circles.
The   angle $\alpha$ at  each vertex of the projection of a  trajectory on the $xy$-plane is constant along
the trajectory.}}
\label{caustics}
\end{center}
\end{figure}

\begin{proposition}[Circular billiard with rough collisions]
For a   billiard system with circular table of radius $r$ and  rough collisions, the projections of trajectories 
from the $3$-dimensional angle-position space to the disc in position plane have the property
that the vertex angle at each collision is a constant of motion. Moreover, for each projected trajectory $\gamma$, there
exists a pair of concentric circles of radius less than $r$
that  are touched  tangentially and alternately by 
 the sequence of
line segments of $\gamma$ at the middle point of these segments. 
\end{proposition}

    \vspace{.1in}
\begin{figure}[htbp]
\begin{center}
\includegraphics[width=4.5in]{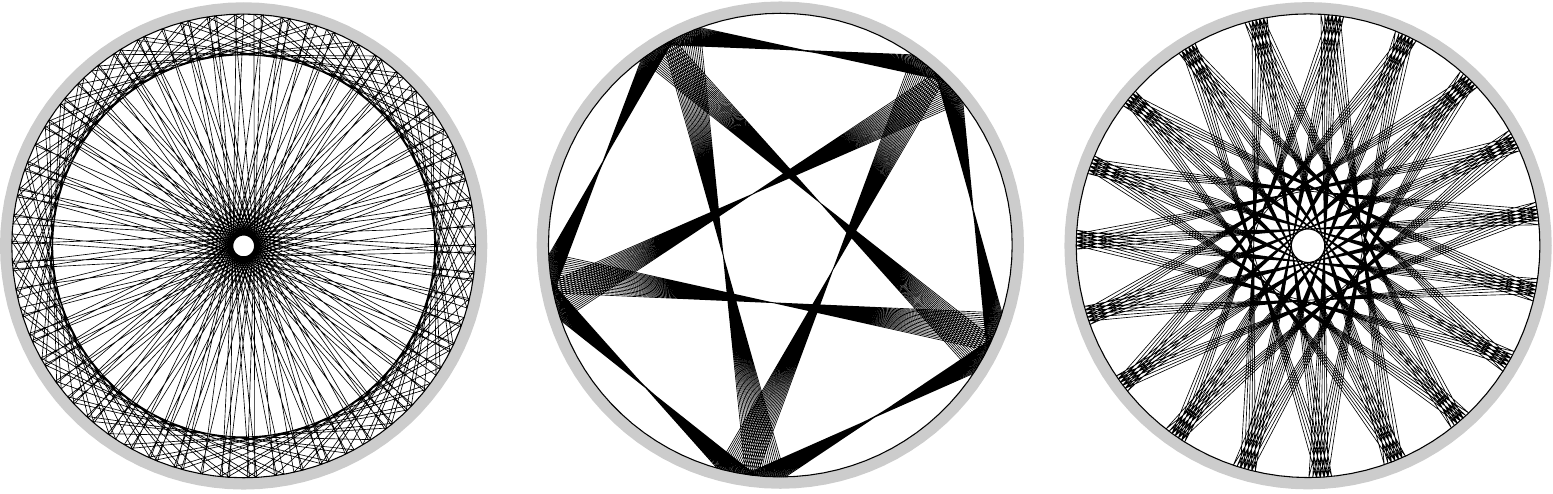}
\caption{\small Three orbit segments with different initial conditions for the motion of the center of mass of a disc in a circular billiard table with rough contact.   }
\label{circle}
\end{center}
\end{figure}

  The next example consists of a moving disc in a wedge-shaped table  with rough collisions.
  A few examples of trajectories for different values of the vertex angle of the billiard table 
  are shown in Figure \ref{periodic}. What is most notable in this case is the existence of bounded orbits.
  Other properties such as periodic orbits for certain angles of the wedge table and caustics are clearly suggested
  by the figures.
  
   \vspace{.1in}
\begin{figure}[htbp]
\begin{center}
\includegraphics[width=4.5in]{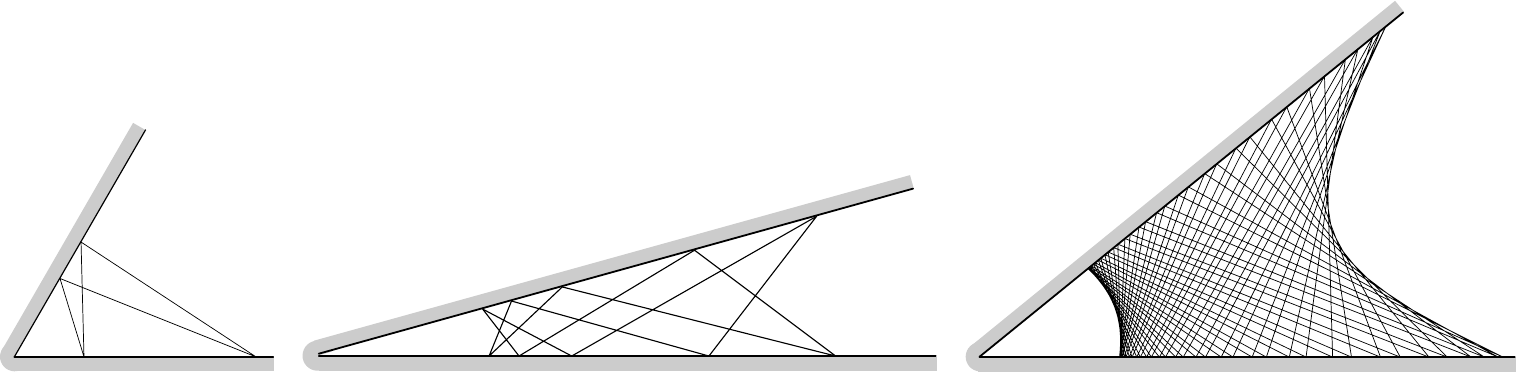}
\caption{\small  Motion of a disc in wedge shaped table. The typical segment of trajectory (more precisely, the motion of
the center of mass only) is shown on the right. The  two trajectories on  the left, which consist of $1000$ free flight segments each,
are likely periodic. }
\label{periodic}
\end{center}
\end{figure}

   In the previous examples the boundary condition on $M$ amounted to a constant (more precise, parallel) choice of $\mathcal{C}_q$.
   We wish to illustrate now boundary conditions for which the map $q\mapsto \mathcal{C}_q$ varies in
   a nontrivial way or is chosen  randomly. 
   Let  the billiard system consist of a disc moving in an infinite strip bounded by two parallel lines. We suppose 
  that one hemisphere  of the   boundary of the disc is rough and the other is smooth.
  In Figure \ref{graph} we show graphs of the position of the (center of) the disc along the longitudinal  axis of the
 table as a function of the collision step. The time between two consecutive collisions is easily shown  to be
 constant, so the step number is proportional to time. The three graphs describe the same trajectory at
 different time scales, as indicated in the legend of the figure.  
    
    \vspace{.1in}
\begin{figure}[htbp]
\begin{center}
\includegraphics[width=4.5in]{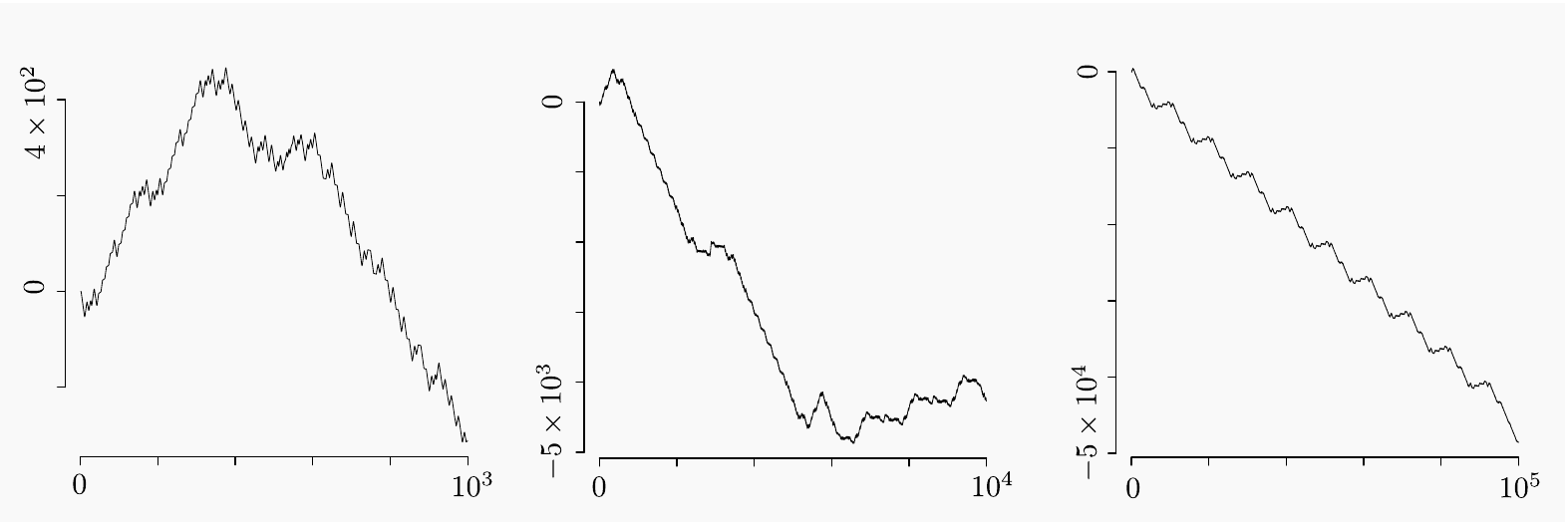}
\caption{\small A single orbit of the motion of a disc between parallel plates in dimension $2$. Half of the boundary circle is rough and the other half is smooth. The
horizontal axis indicates the step number, taken as a proxy for time. The vertical axis gives the  distance of the center of mass along the length
of the $2$-dimensional channel.}
\label{graph}
\end{center}
\end{figure} 

It is interesting to observe the apparent long range quasi-periodic behavior of trajectories. It is also interesting to note
the differences between this example and the next shown in Figure \ref{graphRandom}. The setting is essentially the
same, except that a point on the boundary of the disc is chosen to be rough or smooth randomly with equal probability.
This is thus  an example of a random boundary condition. The longitudinal motion now corresponds to a random walk, 
for which it is possible to prove a diffusion (Brownian motion) limit.

 \vspace{.1in}
\begin{figure}[htbp]
\begin{center}
\includegraphics[width=4.5in]{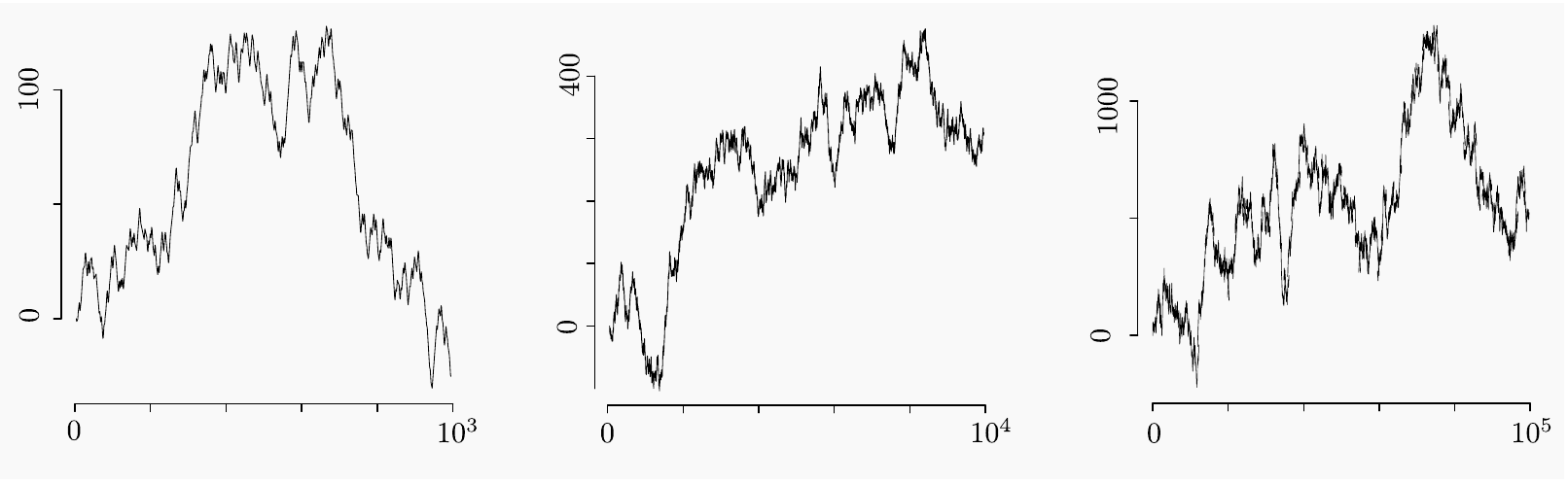}
\caption{\small  Disc between parallel plates in dimension $2$. This is similar to Figure \ref{graph} except that the roughness rank is now random, either $0$ or $1$ with equal probabilities.}
\label{graphRandom}
\end{center}
\end{figure}

We consider now a few examples in dimension $3$. In all cases, a ball of uniform mass
distribution moves  between two parallel infinite plates.  In dimension three, the roughness rank can be $0$, $1$, or $2$.
Standard specular reflection has roughness rank $0$; we explore examples of roughness rank $1$ and $2$.

  \vspace{.1in}
\begin{figure}[htbp]
\begin{center}
\includegraphics[width=4.0in]{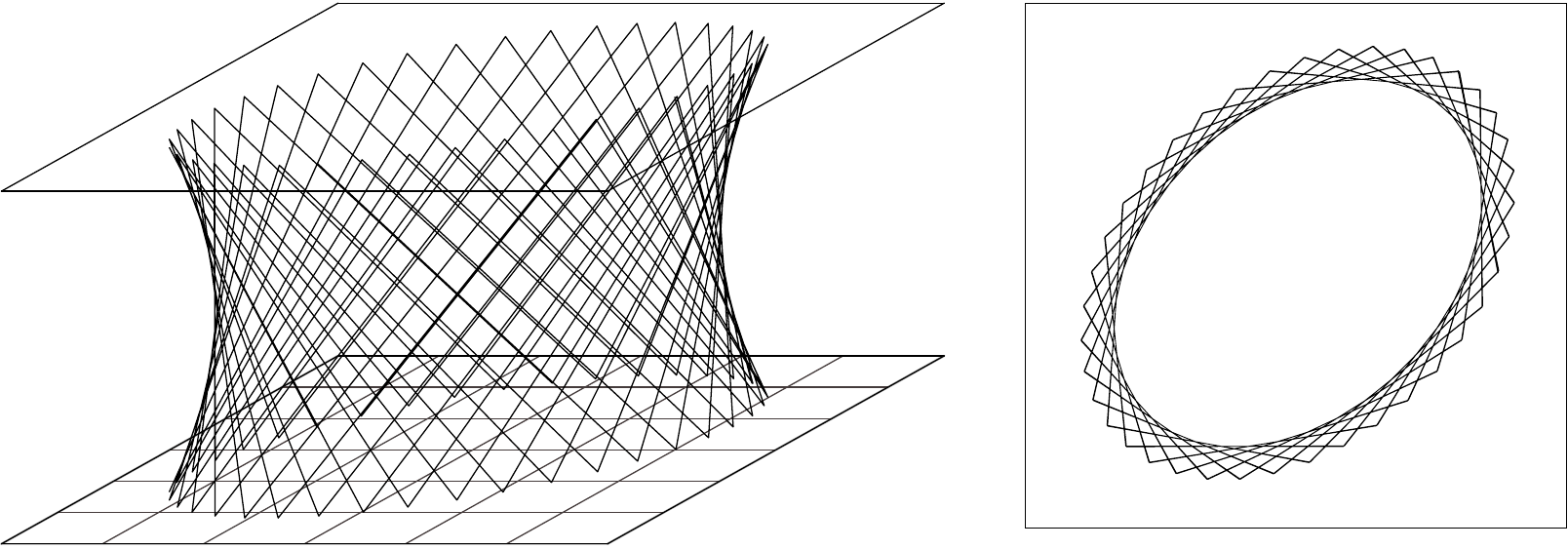}
\caption{\small Typical segment of trajectory in position space  and its horizontal projection  for the motion of the center of mass of a ball bouncing between two parallel plates in
dimension $3$ with roughness rank equal to $2$. Note that orbits are bounded.}
\label{plate}
\end{center}
\end{figure} 

 Figure \ref{plate}
    illustrates the case of roughness rank $2$ collisions. The figure on the right shows the projection of the  trajectory
    to the coordinate plane parallel to the plates. One notable property of the system is that trajectories are
    bounded. In dimension $2$, a similar property was noted in \cite{gutkin}.

  When the roughness rank is one, the set of collision maps comprise a one-dimensional family. 
  Figure \ref{roughrank} shows some combinations of boundary conditions. One clearly notices that
  whenever the roughness rank in at least one plate is not maximal,  trajectories are no longer bounded.
  The boundary conditions for the systems of Figure \ref{roughrank} are of the following types:
  one plate has roughness rank $2$ and the other has roughness rank $1$ with constant rough direction (that is,
  constant map $\mathcal{T}$ in Equation \ref{update}); and
  both plates have roughness rank $1$, with random rough direction for the bottom plate and either constant or
  independent random rough direction for the top plate. Specifically, we choose the directions given by angles $0, \pi/3, 2\pi/3$ with
  equal probabilities. The legend
  of the figure shows which trajectory corresponds to which condition.

    \vspace{.1in}
\begin{figure}[htbp]
\begin{center}
\includegraphics[width=3.0in]{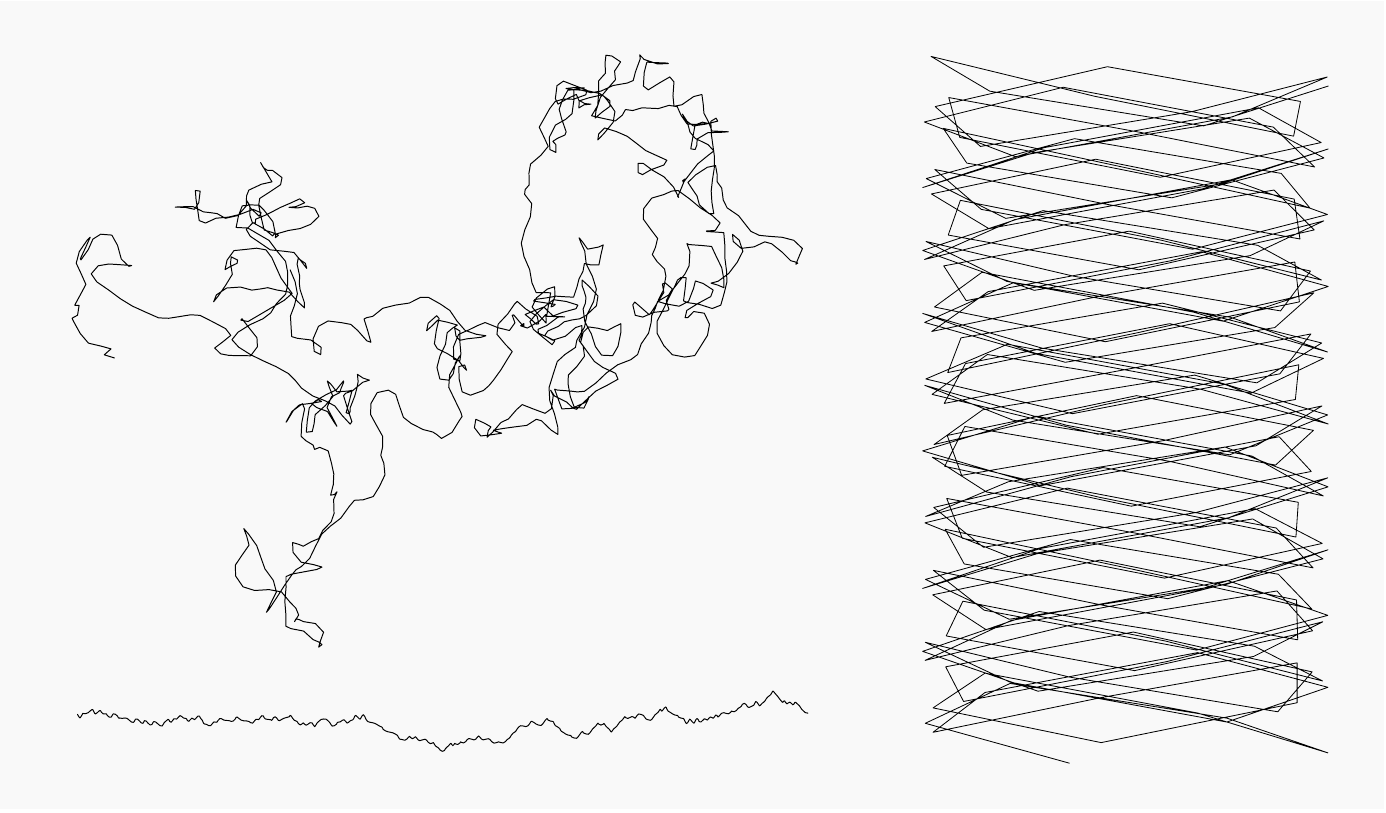}
\caption{\small 
Horizontal projection of motion of the center of mass for a ball bouncing  between two parallel plates in dimension $3$. 
Top left: roughness rank  $1$ for both top and bottom plates and random (and independent) roughness directions; bottom left: top and bottom
roughness rank  $1$, but now roughness direction is constant for the top and random for the bottom plate; right: roughness rank is $2$ for top plate and
$1$ for bottom plate, with constant roughness direction. Making the the rough direction random for the bottom plate
gives a trajectory that does not look significantly different than the one on the top left.}
\label{roughrank}
\end{center}
\end{figure} 
   
 \subsection{invariant measure}
A fundamental property of the dynamics of standard billiard systems is the existence
of a canonical  invariant measure on constant energy surfaces, sometimes referred to as
the Liouville measure. We give here a sufficient condition for the same measure to be
invariant under non-standard collisions. 

Let $S$ denote the boundary of the configuration manifold of the two-bodies system. As before, we assume
that $S$ is smooth. We fix a value $\mathcal{E}$ of the kinetic energy and denote
$$N^{\mathcal{E}}=\left\{(q,v)\in TM: q\in S, \frac12\|v\|^2=\mathcal{E}\right\}. $$
Define the {\em contact} form $\theta$ on $TM$ to be the $1$-form such that $\theta_{v}(\xi)=\langle v, d\pi_{v}\xi\rangle_q$,
where $\pi$ is the base point projection from $TM$ to $M$ (and we indicate the element of $TM$ by $v$ rather than $(q,v)$
in  subscripts). It is well-known that $d\theta$ defines a symplectic form on $TM$. It can also be shown that 
the restriction of $d\theta$ to $N^{\mathcal{E}}$ defines a symplectic form on $N^{\mathcal{E}}\setminus TS$. (See, for example, \cite{cook}.) 

The {\em billiard map} $T$ on $N^{\mathcal{E}}$  associates the post-collision state of 
the system at the time of a collision to the post-collision state at the next collision. 
There are well-known issues about this map, even for standard billiards in dimension $2$,   that make the precise specification of its domain
difficult to describe. See, for example, \cite{chernov}.
Here we assume that the domain of $T$ consists of a ``large'' open set of full Lebesgue in $N^{\mathcal{E}}$ and omit
any further reference to it since this issue of domains is not specific to our rough billiards. The next result is shown by a local argument and
considerations  of domain do not play a role. 

 \vspace{.1in}
\begin{figure}[htbp]
\begin{center}
\includegraphics[width=1.8in]{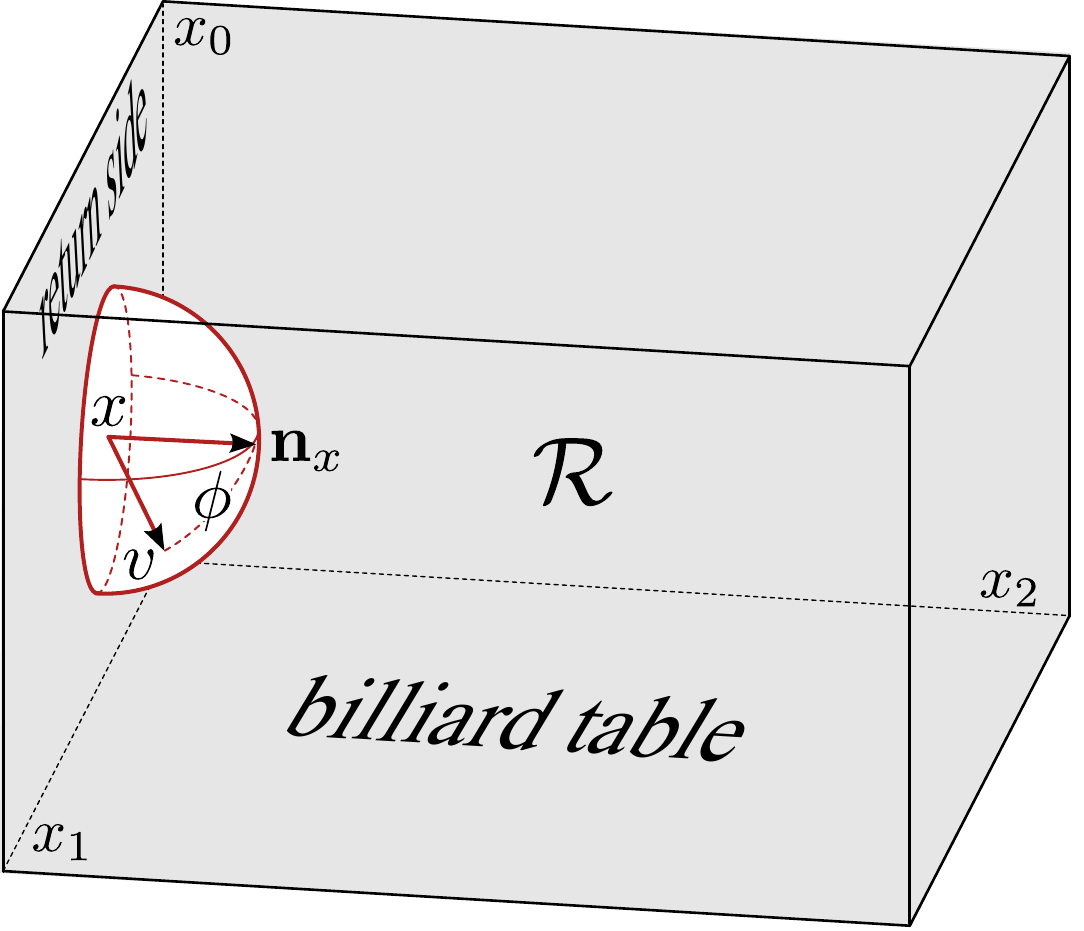}
\caption{\small Angle-position parallelepiped $\mathcal{R}$ for  rectangular billiard table and disc. The canonical invariant billiard
measure on a constant energy hypersurface   is, up to multiplicative constant,  the product of the Euclidean area measure on the boundary of  $\mathcal{R}$ and
the measure on the  hemisphere of velocity directions  given by $\cos\phi \, dA$ where
$dA$ is the Euclidean area measure on the hemisphere. If billiard trajectories are initiated on the side $x_2=0$
with random initial condition given by the just described measure, return trajectories will
have the same distribution. See Figure \ref{histogram}. }
\label{hemisphere}
\end{center}
\end{figure} 

\begin{theorem}\label{invariance}
Suppose that the field of collision maps $q\in S\mapsto \mathcal{C}_q$ is piecewise smooth
and parallel (where it is smooth) with respect to the Levi-Civita connection associated 
to the kinetic energy Riemannian metric. Let $\Omega=d\theta\wedge\cdots \wedge d\theta$
be the    form (of degree $2n-2$, where $n$ is the dimension of the ambient Euclidean
space) derived from the canonical symplectic form 
  $d\theta$ on  $N^{\mathcal{E}}\setminus TS$. Then $\Omega$ is, up to sign, invariant under the billiard map.
\end{theorem}

\begin{corollary}\label{corollary}
Rough billiards in dimension $2$ preserve the canonical billiard measure.
\end{corollary}

The theorem   will be proved in Section \ref{section invariance}. Corollary \ref{corollary} is due  to the following
observation. The boundary of $M$ is a flat surface with the Euclidean metric and the vectors $e_0, J\nu(b)$
shown on the right-hand side of
Figure \ref{interpretation} constitute a parallel frame. The orthogonal line distributions $\mathfrak{C}_q$ and $\mathfrak{S}_q$
are also parallel as the angle between each of them and $e_0$ is constant. But these are the  
eigenspaces of $\mathcal{C}_q$ for the eigenvalues $-1$ and $1$, respectively. It follows that the field of rough collision
maps is parallel. 

We will leave for a future paper   a more detailed investigation of invariant measures
of non-standard billiards. Here we simply illustrate Theorem \ref{invariance} with a numerical observation
concerning the motion of a disc in a rectangular table with rough collisions. 
The geometric set-up is shown in 
Figure \ref{hemisphere}.   The configuration manifold in 
this case  is a parallelepiped $\mathcal{R}$ in dimension $3$ and the canonical billiard measure on the manifold  $N$ of 
unit length vectors  with base points on the boundary of $\mathcal{R}$ has density proportional to 
$\rho(v)=v\cdot \mathbbm{n}_q=\cos \phi$, $0\leq \phi\leq \pi/2$,  with respect to the Riemannian volume measure on $N$, where
$\phi$ is  the angle the vector $v$ makes with the normal vector to the boundary. 

It can be shown that Theorem \ref{invariance} applies to this case. 
As an experiment to illustrate invariance
of the billiard measure for rough collisions we sample initial conditions on the face $x_2=0$ with the uniform
distribution for the $(x_0,x_1)$ positions and  initial velocity  $v$   having probability density proportional to $\cos \phi$ relative
to the uniform probability on the unit hemisphere. 
If the return states to the face $x_2=0$ are  distributed according
to the same measure, then the angle $\phi$ for the return velocity must be distributed relative to Lebesgue measure
on $[0,\pi/2]$ with density $\sin(2\phi)$, which is
the marginal density function for the angle distribution with respect to the Lebesgue measure $d\phi$ on $[0,\pi/2]$, under
the assumption that the billiard measure is invariant.

 \vspace{.1in}
\begin{figure}[htbp]
\begin{center}
\includegraphics[width=2.5in]{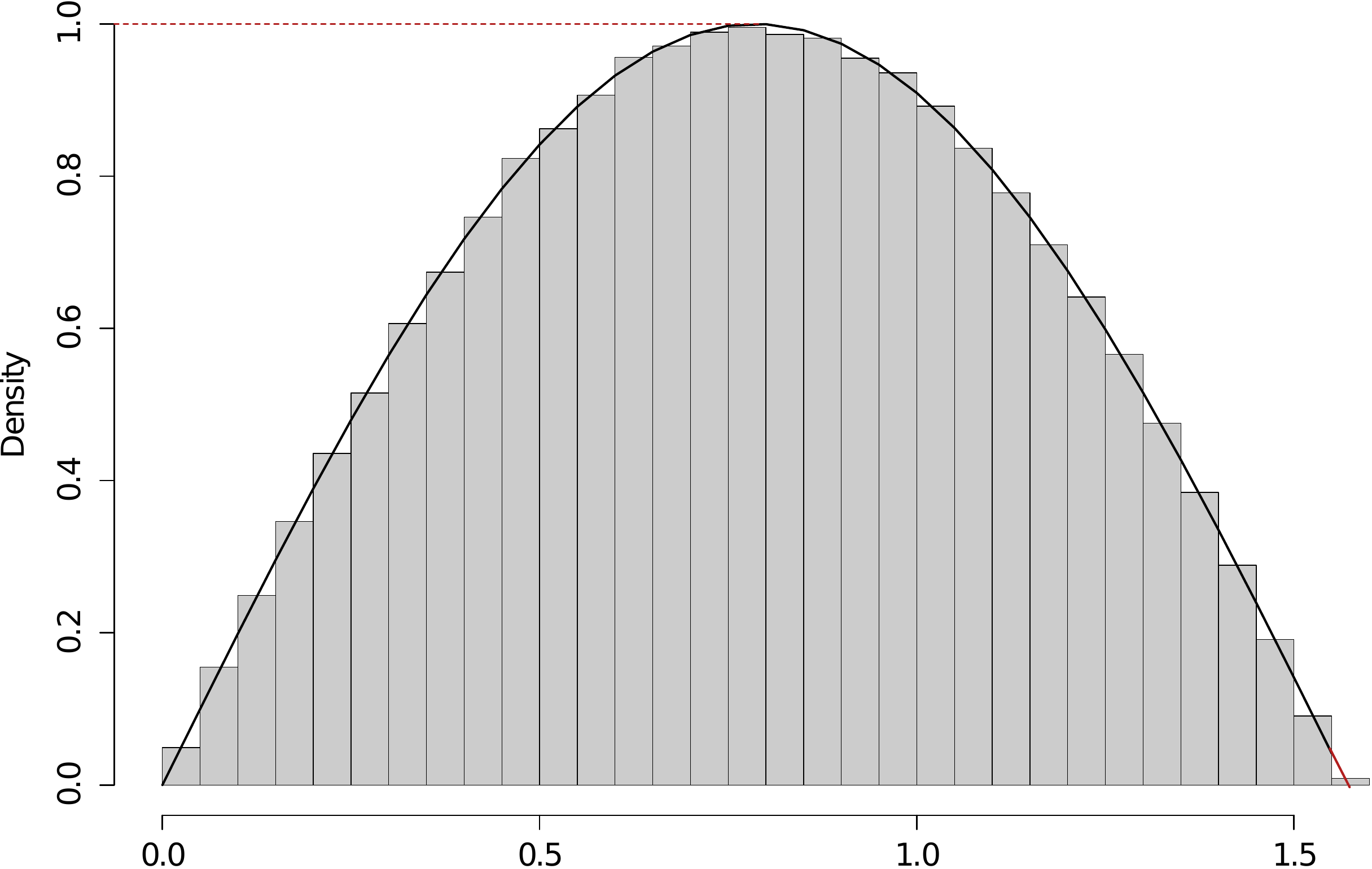}
\caption{\small Experiment to illustrate invariance of the billiard measure for a rectangular table. }
\label{histogram}
\end{center}
\end{figure} 

This is indeed the case as shown in Figure \ref{histogram}. A
 large number ($10^5$) of initial conditions
starting from one side of the rectangle are
sampled from the normalized billiard measure restricted to that side.  For each trajectory, 
the return state to that side is computed and the angle relative to the normal (to the side of
the angle-position parallelepiped corresponding to that side of the rectangle) is recorded. The distribution of  values is
shown in the above histogram. The superimposed line is the graph of   $\sin (2\phi)$.

\section{The Euclidean group and its Lie algebra}
The proofs of the main statements made above will be given after we establish some basic material.
Despite the classical nature of the subject, standard textbook treatments  of collisions of rigid bodies are not
adequate for our needs while the more differential geometric texts in mechanics mostly do not treat this topic.
Thus we find it necessary to develop the subject more or less from scratch. 
In this section we review general facts about the Lie theory and Riemannian geometry of the Euclidean group with left-invariant metrics that will be needed
 thoughout the paper.
 \subsection{Generalities}\label{generalities}
The  {\em  isometry group} of $(\mathbb{R}^n, \cdot)$, where `$\cdot$' indicates the standard inner product, is the Lie group   of all the affine maps of the 
form $x\mapsto Ax+a$ for $A\in O(n)$ and $a\in \mathbb{R}^n$ under  composition of maps. The closed
subgroup of orientation preserving isometries, in which   $A\in SO(n)$,  is  the  {\em Euclidean group} in dimension $n$, denoted   $SE(n)$.  The latter   is 
isomorphic to the semidirect product $ SO(n)\ltimes \mathbb{R}^n$  with
multiplication operation   $$(A_2,a_2)(A_1, a_1)=(A_2A_1, A_2a_1 +a_2)$$
and inverse
$$ (A,a)^{-1}= (A^{-1}, -A^{-1}a).$$
It is also isomorphic to  a subgroup of the general linear group  $GL(n+1,\mathbb{R})$
  under 
the correspondence
  $$(A,u)\in SO(n)\ltimes\mathbb{R}^n\mapsto \begin{pmatrix} A & u\\ 0 & 1\end{pmatrix}
\in GL(n+1,\mathbb{R}).$$
 The Lie algebras of $SO(n)$ and $SE(n)$ will be denoted
 $\mathfrak{so}(n)$ and $\mathfrak{se}(n)$. 
The former consists of
  all the skew-symmetric matrices in the linear space  $M(n,\mathbb{R})$ of $n\times n$ real matrices and 
  $\mathfrak{se}(n)$,   when $SE(n)$ is viewed  as a subgroup of $GL(n+1,\mathbb{R})$, consists of the matrices  $$\begin{pmatrix}X&x\\
0&0 \end{pmatrix} \in M({n+1},\mathbb{R})$$
where $X\in \frak{so}(n)$ 
 and $x$ is any vector in $\mathbb{R}^n$.
 Indicating the matrix by the pair $(X,x)$, the Lie bracket is written
 $$[(X,x),(Y,y)]= (XY-YX, Xy-Yx). $$

 One-parameter subgroups of $SE(n)$ have the form
$$ \sigma(t):=\exp\left(t\begin{pmatrix}X&w\\
0&0 \end{pmatrix}\right)=
\begin{pmatrix}e^{tX}&\int_0^te^{sX}w\,ds \\
0&1 \end{pmatrix}.$$

 It is useful to introduce the {\em wedge product}, the bilinear operation that associates to a pair 
of 
 vectors $a, b$ in $\mathbb{R}^n$ the skew-symmetric matrix $a\wedge b\in\mathfrak{so}(n)$ whose $(i,j)$-entry is
$$(a\wedge b)_{ij}=a_jb_i-a_ib_j.$$
The following elementary properties of the wedge product will be used. 
The transpose of a matrix will be indicated by  $U^\dagger$.
\begin{proposition}\label{exercisewedge}
Let $a, b, u$ be (column) vectors in $\mathbb{R}^n$, $A\in SO(n)$ and  $Z\in \mathfrak{so}(n)$.  Then
\begin{enumerate}
\item $(a\wedge b)u=(a\cdot u) b-(b\cdot u) a  $
\item $(a\wedge b)^\dagger= b\wedge a$
\item $A(a\wedge b)A^{-1}= (Aa) \wedge (Ab)$
\item  $\text{\em Tr}\left((a\wedge b)Z^\dagger\right)= 2 (Za)\cdot b$
\item $\text{\em Tr}\left((a\wedge b)(c\wedge d)^\dagger\right)=(a\cdot c) ( b\cdot d)$
\item Let $V$ be   the span of orthogonal unit vectors $a, b\in \mathbb{R}^n$. Then $(a\wedge b)^2=-I$ and
$$R(\theta):= \exp(\theta a\wedge b)=(\cos\theta) I + (\sin\theta) a\wedge b \in SO(n).$$
Thus $R(\theta)$ is the identity on $V^\perp$, and a rotation on $V$.
\item Let $e_n=(0, \dots, 0,1)^\dagger \in \mathbb{R}^n$ and     $\Pi:\mathbb{R}^n\rightarrow \mathbb{R}^{n-1}=e_n^\perp$   the orthogonal projection. Then  $$Z=\Pi Z\Pi+ e_n\wedge (Ze_n)$$ 
and  $\Pi Z\Pi=0$ iff there exists $z\in \mathbb{R}^{n-1}$ such that $Z= e_n \wedge z$.
\item For   $a\in \mathbb{R}^3$ set $\omega(a)b:=a\times b$\----the 
cross-product by $a$ on the left.  Then
 $a\wedge b=\omega(a\times b)$ and  $A\omega(a)A^{-1}=\omega(A a).$
\item If $n=2$, then $a\wedge b=b\cdot (Ja) J$ where $J$ is  counterclockwise rotation by $\pi/2$. 
\end{enumerate}
\end{proposition}
\begin{proof}
All properties are proved  by  straightforward calculations. 
\end{proof}

   \subsection{Left-invariant Riemannian metrics on $SE(n)$}
  Let $\langle\cdot, \cdot \rangle$ be a left-invariant Riemannian metric on the connected Lie group $G$ with
  Lie algebra $\mathfrak{g}$. Let  $\nabla$ be the associated Levi-Civita
  connection. Define    $B:\mathfrak{g}\times \mathfrak{g}\rightarrow \mathfrak{g}$
  by 
  $$\langle B(u,v),w\rangle = \langle [v,w],u \rangle.  $$
  If $X, Y, Z$ are left-invariant vector fields on $G$ such that $X_e=u, Y_e=v, Z_e=w$, then from 
  \begin{equation}\label{LC0}2\langle \nabla_X Y, Z\rangle = 
  -\langle [Y,Z], X\rangle - \langle [X,Z],Y\rangle +\langle [X,Y],Z\rangle \end{equation}
  we obtain 
\begin{equation}\label{LC}(\nabla_XY)_e = \frac12\left\{[u,v]- B(u,v)-B(v,u)\right\}.\end{equation}
A  left-invariant vector field $X$ is a {\em geodesic} vector field  if and only if
  $0=\nabla_XX=-B(X,X)$. It is   not difficult to show that if 
   the   metric   is bi-invariant then  $B(u,u)=0$ for all $u\in \mathfrak{g}$.

  We adopt the  notation: If $v\in T_gG$,  then $g^{-1}v:=\left( dL_{g^{-1}}\right)_g v\in T_eG=\mathfrak{g}$.
  \begin{proposition}\label{lem1}
  Let  $g(t)$ be  any smooth curve in $G$ and
   $X$  a vector field along $g(t)$, not necessarily left-invariant. Define
  $z(t):=g(t)^{-1}\dot{g}(t)$ and $w(t):= g(t)^{-1}X_{g(t)}$. Then
  $$g(t)^{-1}\left(\frac{\nabla X}{dt}\right)_{g(t)}= \dot{w} +\frac12 \left([z,w]-B(z,w)-B(w,z)\right).$$
  In particular, $g(t)$ is a geodesic  if and only if 
  $\dot{z}=B(z,z). $
  \end{proposition}
  \begin{proof}
  Let $e_1, \dots, e_n$ be a basis of $\mathfrak{g}$ and $E_1, \dots, E_n$  the respective  left-invariant vector fields on $G$.
  We  write  $X=\sum f_j E_j$, $g'(t)=\sum h_j(t) E_j(g(t))$. Then, using Equation \ref{LC},
  $$ \frac{\nabla X}{dt} = \sum_{j,k}h'_j(t) \left[ E_jf_k + \frac12 f_k\left([E_j,E_k]-B(E_j,E_k)-B(E_k,E_j)\right)\right]_{g(t)},$$
from which we obtain the desired expression after left-multiplying by $g(t)^{-1}$.
  \end{proof}

  Let the Lie algebra of $G$ be $\mathfrak{g}=\mathfrak{s}\oplus \mathfrak{r}$, where
  $\mathfrak{r}$ is an ideal and $\mathfrak{s}$ is a Lie subalgebra. 
  Let $\langle \cdot,\cdot\rangle$ be a left-invariant Riemannian metric on $G$ and $\nabla$ the corresponding Levi-Civita
  connection. We suppose that $\mathfrak{s}$ and $\mathfrak{r}$ are orthogonal subspaces.
  
  \begin{proposition}
  The following properties hold, where we indicate by the same letter elements of $\mathfrak{g}$ and the
  associated left-invariant vector fields on $G$. For $z, z_j\in \mathfrak{r}$, $Z, Z_j\in \mathfrak{s}$
  \begin{enumerate}
  \item $\nabla_{Z_1}Z_2$ and $ B(Z_1,Z_2)$ lie in $ \mathfrak{s}$
  \item $\nabla_{z_1}z_2\in \mathfrak{r}$. 
If $\mathfrak{r}$ is abelian,  $\nabla_{z_1}z_2=0$ and $B(z_1, z_2)\in \mathfrak{s}$. 
\item $\nabla_zZ=0$ and $\nabla_Zz=[Z,z]$. Moreover $B(z,Z)\in \mathfrak{r}$ and $B(Z,z)=0$. 
  \end{enumerate}
  \end{proposition} 
  \begin{proof}
  All properties follow  from the definition of $B$, Expression \ref{LC0} for  the Levi-Civita connection,
  and the assumption that the subalgebra $\mathfrak{s}$ and the ideal $\mathfrak{r}$ are orthogonal.
  \end{proof}
  
  Let $S$ and $R$ be the subgroups of $G$ having lie algebras $\mathfrak{s}$ and $\mathfrak{r}$, respectively. 
  Then $G$ is the semi-direct product $G=S\ltimes R$, where $R$ is a normal subgroup of $G$. We now assume that $R$ is
  a vector subgroup, hence abelian, and that $S$ is a compact subgroup  acting   on $R$ by linear transformations, $S\subset GL(R)$,
  preserving an inner product $\langle\cdot, \cdot\rangle$ on $R$. That is, $S$ is a 
  subgroup of the orthogonal group $O(R, \langle\cdot,\cdot\rangle)$.  Elements of $G$ will
  be denoted $(A,a)$ where $A\in S$ and $a\in R$.  Indicating the action of $S$ on $R$   by $Aa$, the
  multiplication in $G$ takes the form
  $$(A_1,a_1)(A_2,a_2)=(A_1A_2, A_1a_2+a_1). $$
  Note that $\text{Ad}_{(A,0)}(0,z)=(0,Az)$ and  $\text{ad}_{(Z,0)}(0,z)= (0,Zz)$, where $\langle Zz, z\rangle =0$ since
  $A$ acts on $R$ by isometries. 
  \begin{proposition}Under the just stated assumptions
  $B(z,z)=0$ and 
  $B(z,Z)=-Zz$.  If $g(t)$ is a geodesic, 
  writing $(Z(t), z(t))=g(t)^{-1}\dot{g}(t)$ we have
$
  \dot{Z}=B(Z,Z)$ and 
$  \dot{z}=-Zz$
   \end{proposition}
  \begin{proof}
  These simple remarks are  consequences of the definition of $B$, the algebraic assumptions about the group,
  and Proposition \ref{lem1}.
  \end{proof}

  \section{Newtonian mechanics of  rigid bodies}
  We give here  some alternative expressions of Newton's equation of motion. The approach,
  if not the notations, is essentially that of \cite{arnold}. Other useful references are \cite{bloch} and \cite{marsden}. 
  
  \subsection{Momentum of a tangent vector and the momentum map}
  
  If $M$ is a Riemannian manifold with metric $\langle\cdot, \cdot\rangle$ and $v\in T_qM$, we  denote
  $$ \mathcal{P}(q, v):=\langle v, \cdot\rangle_q\in T_q^*M$$
  and call this covector the {\em momentum} associated to the (velocity) vector $v$ at (configuration) $q$.
  The pair $(q,v)$ will be called a {\em state} of the system.
  We often indicate states by $(q,\dot{q})$, dotting the quantities of which time derivative is taken.
The momentum map
 $\mathcal{P}^\mathfrak{g}:TM\rightarrow \mathfrak{g}^*$ was defined in Section \ref{definitions mechanics}.

   If $M=G$ is endowed with a left-invariant Riemannian metric, the momentum map for  the left-action of $G$
  is given by
  $ \mathcal{P}^{\mathfrak{g}}(g,\dot{g})(u)= \langle \dot{g}, \left(dR_g\right)_e u\rangle_g.$
 Because the Riemannian metric
  is left-invariant,  
  $$ \mathcal{P}^{\mathfrak{g}}(g,\dot{g})(u)= \langle g^{-1}\dot{g}, \text{Ad}_{g^{-1}} u\rangle_e.$$
  In this case  we also write $\mathcal{P}^\mathfrak{g}(v):=\langle v,\cdot\rangle_e\in \mathfrak{g}^*$
  for $v\in \mathfrak{g}$.
  Then
  $$ \mathcal{P}^\mathfrak{g}(g,\dot{g})= \text{Ad}^*_{g^{-1}} \mathcal{P}^\mathfrak{g}(g^{-1}\dot{g})$$
  where $\text{Ad}^*_g\alpha=\alpha\circ \text{Ad}_g$ and $\text{Ad}_g$ is the differential of the map  $L_g\circ R_{g^{-1}}$.

 \begin{proposition}\label{lem2}
 Given a smooth curve $g(t)$ in $G$ and setting  $z(t):=g(t)^{-1}\dot{g}(t)$, then 
$$ \frac{d}{dt}\mathcal{P}^{\mathfrak{g}}(g,\dot{g})  =\text{\em Ad}_{g^{-1}}^*\mathcal{P}^{\mathfrak{g}}(\dot{z}-B(z,z)).$$
In particular, $g(t)$ is a geodesic if and only if momentum $\mathcal{P}^\mathfrak{g}(g,\dot{g})$ is constant. 
 \end{proposition}
  \begin{proof}
 First note that 
 $$ \frac{d}{dt} \text{Ad}_{g^{-1}} (u) = -[z,\text{Ad}_{g^{-1}}u].$$
 It follows from the definitions that
 $$\frac{d}{dt}\mathcal{P}^\mathfrak{g}(g,\dot{g})(u) = \frac{d}{dt}\langle z,\text{Ad}_{g^{-1}}u\rangle=\langle \dot{z}, \text{Ad}_{g^{-1}}u\rangle - \langle z, [z,\text{Ad}_{g^{-1}}] \rangle= \langle \dot{z}-B(z,z), \text{Ad}_{g^{-1}}u\rangle. $$
 The expression on the far right is now $\mathcal{P}^{\mathfrak{g}}(\dot{z}-B(z,z))\circ \text{Ad}_{g^{-1}}$ evaluated at $u$. 
  \end{proof}

When  $G=SO(n)$, 
 define  on $M(n,\mathbb{R})$   the bilinear form
  $$\langle X, Y\rangle_0:=\text{Tr}(XY^\dagger). $$
  Then  $\langle\cdot,\cdot\rangle_0$ is  an $\text{Ad}_G$-invariant non-degenerate positive
  bilinear form on $\mathfrak{so}(n)$ and the associated left-invariant  Riemannian metric on $G$ is bi-invariant. 
  Thus for any left-invariant Riemannian metric $\langle\cdot, \cdot\rangle$ there must exist a linear
  map $\mathcal{L}:\mathfrak{g}\rightarrow\mathfrak{g}$, symmetric and positive definite with respect to
  $\langle\cdot, \cdot\rangle_0$, such that $\langle Z_1, Z_2\rangle=\frac12\langle \mathcal{L}(Z_1), Z_2\rangle_0.$ 
  We are interested in   such $\mathcal{L}$ that arises   from a symmetric 
    matrix $L\in M(n,\mathbb{R})$ according to the definition
    $\mathcal{L}(Z):=ZL+LZ$, in which case 
     $$  \frac12 \text{Tr}(\mathcal{L}(Z_1)Z_2^\dagger)=\text{Tr}\left( Z_1LZ_2^\dagger\right).$$
   If $u_1, \dots, u_n$ is a basis of $\mathbb{R}^n$ of eigenvectors of $L$,
    $Lu_i=\lambda_i u_i$,
    then the $u_i\wedge u_j$ comprise a basis of $\mathfrak{so}(n)$  such that
    $\mathcal{L}(u_i\wedge u_j)=(\lambda_i+\lambda_j) u_i\wedge u_j.$

  \begin{proposition}    Given 
 $L\in M(n,\mathbb{R})$  
  define the
linear map $ \mathcal{L}: M(n,\mathbb{R})\rightarrow  M(n,\mathbb{R})$ 
 by $\mathcal{L}(Z)=ZL+LZ$. If $L$ is symmetric and non-negative  definite of rank 
 at least $n-1$, then  $\mathcal{L}$ is an isomorphism and 
 $\langle Z_1,Z_2\rangle:= \frac12 \text{\em Tr}(\mathcal{L}(Z_1)Z_2^\dagger)$ is a left-invariant Riemannian metric
 on $SO(n)$. The tensor $B$  for this metric is
 $$ B(Z_1, Z_2)= \left[LZ_1, Z_2\right]L^{-1}$$
  for all $Z_1, Z_2\in \mathfrak{so}(n)$. 
\end{proposition}
\begin{proof}
Let $\lambda_1, \dots, \lambda_l$ be the distinct eigenvalues  and  $V_1, \dots, V_l$ the respective  eigenspaces of $L$.
Let  $\pi_j:\mathbb{R}^n \rightarrow V_j$ denote the orthogonal projections. Then $\pi_jL = L\pi_j =\lambda_j\pi_j$. 
It suffices to show that $\mathcal{L}$ has trivial kernel. Thus suppose  $\mathcal{L}(Z)=0$. Then for all $i,j$, 
$$ 0=\pi_i \mathcal{L}(Z)\pi_j=\pi_i LZ\pi_j+\pi_i ZL\pi_j= (\lambda_i+\lambda_j)\pi_iZ\pi_j.$$
But  $\lambda_i+\lambda_j>0$
by the assumptions  on $L$  so  all  blocks $\pi_iZ\pi_j$ are zero, hence $Z=0$. 
The  expression for $B$   follows from  
    $ \text{Tr}\left([Z_2, Z_3]LZ_1^\dagger\right)=\text{Tr}\left([LZ_1, Z_2]L^{-1} Z_3^\dagger\right).$
\end{proof}

  \subsection{Kinetic energy metrics on $SE(n)$ for rigid bodies in $\mathbb{R}^n$}     
The left-invariant metrics on $G=SE(n)$ of interest here are derived from  mass distributions on 
the rigid body.
  Let $B\subset\mathbb{R}^n$ denote the body in its {\em reference configuration}.
  The position of  {\em material point} $b\in B$ in the configuration $g=(A,a)\in G=SO(n)\times \mathbb{R}^n$
   is $\Phi(g,b):=Ab+a$. 
  We call $\Phi:G\times B\rightarrow \mathbb{R}^n$ the {\em position map}
  and use the alternative notations $\Phi(g,b)=g(b)=\Phi_b(g)$ as convenience dictates. 
  For now (until we consider collisions shortly) $B$ may be any 
  measurable set with a finite (positive) measure $\mu$ defining its mass distribution.
  Recall from Section \ref{definitions mechanics} that 
  $m:=\mu(B)$ is the mass of the body and the first moment of $\mu$ is $0$.
   When considering the motion of several bodies, we assume that the  center of mass of each of them in
 the standard configuration is at $0$.
  
  Elements of   $\mathfrak{g}=\mathfrak{so}(n)\times \mathbb{R}^n$ will be written  in the form $\xi=(Z,z)$. 
  Let  $L_g$ and $R_g$ denote left and right-multiplication by $g$. We will very often use the  identification
 $G\times \mathfrak{g}\cong TG$ given by $(g,\xi)\mapsto (dL_g)_e\xi$.  
Each $v\in T_gG$  gives rise to the map $V_v:B\rightarrow \mathbb{R}^n$   defined by
$V_v(b)=\left(d\Phi_b \right)_q v$, which is the velocity of $b$ in state  $(g,v)$.
The {\em kinetic energy Riemannian metric}
on $G$ is defined so that the  inner product of $u, v\in T_gG$ is given by
$$\langle u,v\rangle_g= \int_{B} V_u(b)\cdot V_v(b)\, d\mu(b).$$  
\begin{proposition}\label{Vb}
The Riemannian metric on $SE(n)$ associated to the mass distribution $\mu$ is invariant under left-translations. 
\end{proposition}
\begin{proof}
To see this, note first that  
$$V_v(b)=\left.\frac{d}{ds}\right|_{s=0} g e^{s\xi} b=\left.\frac{d}{ds}\right|_{s=0}\left(A e^{sZ}b + A\int_0^s e^{tZ}z\, dt  + a\right)= A\left(Zb+z\right).$$
Here we have used the form of the exponentiation in $SE(n)$ given in Section \ref{generalities}.  Therefore, as $A$ leaves invariant the standard inner product
in $\mathbb{R}^n$,  
 $$V_u(b)\cdot V_v(b)= (Z^ub+z^u)\cdot(Z^vb+z^v)$$
and so
$ \langle (dL_g)_e\xi, (dL_g)_e\eta\rangle_g =\langle \xi,\eta\rangle_e.$
\end{proof}

Recall the {\em inertia matrix} $L$ introduced in Section \ref{definitions mechanics}.
\begin{proposition}\label{inertiapropo}
The  matrix $L$  associated to mass distribution   $\mu$   satisfies:
\begin{enumerate}
\item For arbitrary $n\times n$ matrices
$Z_1$ and $Z_2$,
$\int_{B} (Z_1b) \cdot (Z_2b)\, d\mu(b) = \text{\em Tr}\left(Z_1 L Z_2^\dagger\right). $
\item If $A\in SO(n)$,  the inertia matrix of the rotated  body   $gB$ is
$ L^A:= A L A^\dagger.$
\item  $L=\lambda I$ if   $\mu$ is $SO(n)$-invariant.
   If $\mu$ is uniform on a ball of radius $R$,  
$\lambda= (n+2)^{-1}{R^2}$.
\end{enumerate}
\end{proposition}
 \begin{proof}
These are obtained by elementary calculations.
 \end{proof}

    Let $\mathcal{L}(Z)=LZ+ZL$ where, from now on, $L$ is an inertia matrix. 

\begin{corollary} \label{expression of metric} The kinetic energy Riemannian metric can be written in the form
\begin{equation}\label{riemannian} \langle u,v\rangle_g=m\left[\frac12\text{\em Tr}\left(\mathcal{L}(Z_u)Z_v^\dagger\right) + z_u\cdot z_v\right]\end{equation}
where   $u,v\in T_gG$ and   their   left-translates to $\mathfrak{g}$  are indicated by
$(Z_u,z_u)$ and $(Z_v,z_v)$.  \end{corollary}

  \begin{proposition}[Tensor $B$ for $\mathfrak{se}(n)$] 
     Let $SE(n)$ be given the left-invariant Riemannian metric associated to the inertia matrix $L$. 
     Then
     $$B((Z_1,z_1),(Z_2,z_2))=\left(\left([LZ_1, Z_2] -\frac12 z_1\wedge z_2\right)L^{-1}, -Z_2z_1\right). $$
     \end{proposition}
\begin{proof}
Observe that 
\begin{align*}
\left\langle [(Z_2,z_2),(Z_3, z_3)], (Z_1,z_1) \right\rangle&=\left\langle ([Z_2,Z_3], Z_2z_3-Z_3z_2), (Z_1,z_1) \right\rangle\\
&= m\left\{ \text{Tr}\left([Z_2,Z_3]LZ_1^\dagger\right) + (Z_2z_3-Z_3 z_2)\cdot z_1\right\}\\
&= m\text{Tr}\left( \left([LZ_1,Z_2] -  \frac12 z_1\wedge z_2\right) L^{-1} LZ_3^\dagger\right) -m(Z_2 z_1)\cdot z_3\\
&=\left\langle \left( \left( [LZ_1,Z_2] -\frac12 z_1\wedge z_2\right)L^{-1}, -Z_2z_1\right), (Z_3,z_3) \right\rangle.
\end{align*}
The claimed identity now follows from the definition on $B$. 
\end{proof}

We note  that if  $L=\lambda I$, then $B((Z,z),(Z,z))=(0,-Zz)$.

 \begin{proposition}\label{propomom}
 Give $G=SE(n)$ the left-invariant Riemannian metric defined by
a mass distribution on the rigid body $B$ with   
  inertia matrix $L$ and mass $m$.  Let $\xi=(W,w)\in \mathfrak{g}$ and $(g,v)\in TG$
where $g=(A,a)$ and $v=(dL_g)_e(Z,z)$. Then  
$$\mathcal{P}^\mathfrak{g}(g,v)(\xi)=\frac12 m\text{\em Tr}\left\{\left(\text{\em Ad}_A \mathcal{L}(Z)+  x_c\wedge v_c \right) W^\dagger\right\} + mv_c \cdot w. $$
Here $x_c=a$ is the position of the center of mass of the body  in configuration $g$ and $v_c:=Az$ is the 
velocity of the center of mass for the given state $(g,v)$. 
 \end{proposition}
\begin{proof}
This is a straightforward computation based on the definition of $\mathcal{P}^\mathfrak{g}$ and the expression
of the Riemannian metric given in Corollary \ref{expression of metric}.
\end{proof}

  Let $\langle\cdot, \cdot\rangle_\mathfrak{g}$ be the left-invariant inner product on $\mathfrak{g}$ given by
\begin{equation}\label{innerproductLieg} \langle (Z,z),(W,w)\rangle_\mathfrak{g}:= \text{Tr}(ZW^\dagger) + z\cdot w.\end{equation}
   Then, with the notation of Proposition \ref{propomom},
\begin{equation}\label{usefulmom} \mathcal{P}^\mathfrak{g}(g,v)(\xi)= m\left\langle  \left(
    \frac12 \left(
    \text{Ad}_A\mathcal{L}(Z) + x_c   \wedge  v_c  \right), v_c\right), (W,w) 
    \right\rangle_\mathfrak{g}
    \end{equation}

\subsection{Singular force fields and impulses}
 Let  $M$ be  
 the configuration manifold of a mechanical system with the  kinetic energy Riemannian metric
and  material body ${B}$ with  mass distribution measure $\mu$.

  A {\em force field} on $M$ is a bundle map $F:TM\rightarrow T^*M$
 possibly depending on time, although we omit explicit reference to the time variable. So if $v\in T_qM$, then $F(q,v)\in T_q^*M$. Forces acting on   ${B}$
 typically arise from a $\mathbb{R}^n$-valued (possibly time dependent)   measure $\varphi$   on  ${B}$, parametrized by
 $TM$, called the {\em force distribution}. From such a measure we   define the force field  $F(q,v)\in T_q^*M$ such that
for each $u\in T_qM$, 
 $$ F(q,v)(u)=\int_{B} V_u(b)\cdot d\varphi_{q,v}(b).$$
We are  interested in cases where  $\varphi_{q,v}$ is  singular,  supported on a single point of ${B}$.

\begin{definition}[Newton's equation]
Newton's equation of motion of the (unconstrained) mechanical system with configuration manifold $(M,\langle\cdot,\cdot\rangle)$
and force field $F$ is 
$$\frac{\nabla }{dt} \mathcal{P}(q,\dot{q}) = F(t, q, \dot{q}). $$
\end{definition}

\begin{proposition}\label{newtons}  
Give $M=SE(n)$ a left-invariant Riemannian metric
and let $F$ be a force field on $M$. 
 Let $F^\#$ be the  dual field, so $F(t,q,v)(u)=\langle F^\#(t,q,v),u\rangle_q$. Then
$$\frac{\nabla }{dt} \mathcal{P}(g,\dot{g}) = F \Leftrightarrow \frac{\nabla \dot{g}}{dt} = F^\# \Leftrightarrow
\dot{w}-B(w,w) = (dL_{g^{-1}})_gF^\# \Leftrightarrow \frac{d}{dt}\mathcal{P}^{\mathfrak{g}}(g,\dot{g}) =  R_{g}^*F.$$
  \end{proposition}
  \begin{proof}
These are  consequences of Propositions \ref{lem1} and \ref{lem2}.
  \end{proof}

   When $M=SE(n)$,
  it is useful to regard the force field as a Lie algebra-valued  by left-translating each force vector to
  $T_eG$. We define
  $$\mathcal{F}(t,g,\dot{g}):= (Y(t, g, \dot{g}), y(t,g,\dot{g})):=(dL_g)_e^{-1}F^\#(t, g, \dot{g})\in \mathfrak{g}. $$
  Then, using the notation of \ref{innerproductLieg}, 
  \begin{equation}\label{form of F}
  \begin{aligned}
   \langle \mathcal{F}, \text{Ad}_{g^{-1}} \xi\rangle_\mathfrak{g}&=\frac12 \text{Tr}\left[\left(\text{Ad}_A\mathcal{L}(Y) + x_c\wedge Ay \right)W^\dagger\right]+ (Ay)\cdot w\\
   &=  \left\langle \left( \frac12\left(
    \text{Ad}_A\mathcal{L}(Y) + x_c\wedge   Ay\right),Ay\right), (W,w)\right\rangle_\mathfrak{g}
   \end{aligned}
   \end{equation}
  Proposition \ref{newton2} follows from these remarks, keeping in mind that $v_c=Az$ and $x_c=a$.

Let the force be   applied   on a single point 
  $Q=Q(t,g, \dot{g})\in \mathbb{R}^n$, so that $\varphi$  is supported on $Q$.
Let  $u=(dR_g)\xi\in T_qM$, $\xi=(W,w)\in \mathfrak{g}$. Note that we are using right-translation here since
we wish to evaluate the last of the equivalent equations of Proposition \ref{newtons} on the Lie algebra element $\xi$. Then
 there exists $\mathcal{I}=\mathcal{I}(t,g,\dot{g})$ depending only of $F$ such that 
\begin{equation*}F(t, g, \dot{g})(u)=\int_{B} V_u(b)\cdot d\varphi_{g,\dot{g}}(b)
=V_u(g^{-1}Q)\cdot \mathcal{I}
=(WQ +w)\cdot \mathcal{I}
=\frac12 \text{Tr}((\mathcal{I}\wedge Q) W^\dagger) + \mathcal{I}\cdot w.
\end{equation*}
This gives
$$F(t, g, \dot{g})(u) =\left\langle \left(\frac12  Q\wedge \mathcal{I}, \mathcal{I}\right), (W,w)\right\rangle_\mathfrak{g}.$$
 It follows from the expression \ref{form of F} of $F$
  that 
 \begin{align*}
\text{Ad}_A\mathcal{L}(Y) +   x_c  \wedge  Ay&=  Q\wedge \mathcal{I}\\
    Ay&=\mathcal{I}.
 \end{align*}
 Writing $f_c:=Ay$, this is equivalent to $f_c=\mathcal{I}$ and 
$
\text{Ad}_A\mathcal{L}(Y)  =   (Q-x_c) \wedge   f_c  .
$
 Therefore, the equation of motion becomes
 \begin{align*}
 m\dot{v}_c&=f_c\\
 m\,\text{Ad}_A \left(\mathcal{L}\dot{Z} -[\mathcal{L}Z,Z]\right)&=(Q-x_c)\wedge   f_c
 \end{align*}
  proving Proposition \ref{proposition two equations}.
  
 Our  informal discussion of the idea of impulse from earlier in the paper and the above remarks
 now give the expression $((Q-x_c)\wedge \mathcal{I}_c, \mathcal{I}_c)\in \mathfrak{g}$ for the change in momentum 
 due to singular forces  applied to $Q$.  This gives the following.
 
 \begin{proposition}[Change in momentum due to impulsive forces]\label{prop prop}
If the rigid body with mass $m$, inertia matrix $L$ and associated Lie algebra map $\mathcal{L}$, 
 is subject to an impulsive force concentrated at point $Q\in \mathbb{R}^n$ at a given time, then  momentum changes discontinuously
 according to 
 \begin{equation}\label{imp}
 \begin{aligned}
 mv_c^+ - mv_c^-&= \mathcal{I}_c\\
 m\,   \text{\em Ad}_A\mathcal{L}(Z^+-Z^-) &= (Q-x_c)   \wedge \mathcal{I}_c
 \end{aligned}
\end{equation}
for some vector $\mathcal{I}_c\in \mathbb{R}^n$ depending on the state of the body.
As before, $x_c=a$ indicates the center of mass of the body in configuration $g=(A,a)$,  $v^\pm=Az^\pm$
are the velocities  of the center of mass immediately prior to and after the application of the impulse, and
$(g, (Z^\pm, z^\pm))$ are the pre- and post-impulse states of the body. 
 \end{proposition}

\subsection{Several bodies and momentum conservation}\label{several}
If  the mechanical system consists of several unconstrained rigid bodies,
$B_1,\dots, B_k$ (in reference configuration)
subject to forces $F_j(q,v)$, $j=1, \dots k$,  the configuration manifold $M$ is a subset of the product 
$G\times \cdots \times G$, with one copy of $G=SE(n)$ for each body. We consider for now only
motion in the interior of $M$.

We say that forces are  {\em internal} to the system if they are somehow 
due to the influence of the bodies on each other. More specifically, we use term `internal' 
when $F_j=\sum_{i\neq j} F_{ij}$ and the $F_{ij}=F_{ij}(q,v)$\----the force body $i$ exerts on body $j$ in state $(q,v)$\----satisfies the property of {\em action-reaction}: $F_{ij}=-F_{ji}$.
If the forces are derived from, possibly singular,  measures  $\varphi_{q,v}(j, b|i,b')$ on $B_i\times B_j$ 
  so that
$$  F_{ij}(q,v)(u)=\int_{B_i}\int_{B_j} V_u(b')\cdot d\varphi_{q,v}(j,b'|i,b),$$
then the action-reaction property, expressed in terms of $\varphi$, means
that
$$d\varphi(i,b|j,b')=-d\varphi(j,b'|i,b)$$
 for almost every $b, b'$ (with respect to $\varphi$).
Newton's equation applied to body $j$ is then
$$ \frac{\nabla }{dt} \mathcal{P}_j(q,\dot{q}) = \sum_{i\neq j}F_{ij}(t, q, \dot{q})$$
and the total momentum  $\mathcal{P}(q,\dot{q})=\sum_j\mathcal{P}_j(q,\dot{q})$ is conserved:
$$\frac{\nabla }{dt} \mathcal{P}(q,\dot{q}) = \sum_j \sum_{i\neq j} F_{ij}(t,q,\dot{q}) =0.$$

 Another  way to interpret the notion of forces  internal to the system is to 
 assume that the total work  the  $F_i$ do along a rigid motion of the entire system, that is, the work 
 along a path in $M$ of the form
 $\gamma(t):= e^{t\xi}q= (e^{t\xi}g_1, \dots, e^{t\xi}g_k)$
 is zero.
 The total work is then
\begin{align*}0&=\int_a^b \sum_{j}F_j(\gamma(t),\gamma'(t))(\gamma_j'(t))\, dt\\
&=\sum_j \int_a^b (R_g^*F_j(\gamma(t),\gamma'(t)))(\xi)\, dt\\
&= 
\sum_j \int_a^b \frac{d}{dt}\left[\mathcal{P}^\mathfrak{g}_j(\gamma_j(t),\gamma_j'(t))(\xi)\right]  dt\\
&=\sum_j\mathcal{P}^\mathfrak{g}_j(\gamma(b), \gamma'(b))(\xi) - \sum_j\mathcal{P}^\mathfrak{g}_j(\gamma(a),\gamma'(a))(\xi)
\end{align*}
  and, again, the total momentum (now in the sense of the momentum map on $\mathfrak{g}$)
  is constant. In this sense, conservation of momentum follows, as expected, 
from a symmetry property. 
 
 Of particular interest here are  two bodies  that
 interact through impulses  applied to a common point of collision $Q$.  Then for each body, indicated by
 the index $i=1,2$,
 \begin{equation}\label{imp}
 \begin{aligned}
 m_iv_{c,i}^+ - m_iv_{c,i}^-&= \mathcal{I}_{c,i}\\
 m_i\,   \text{Ad}_{A_i}\mathcal{L}_i(Z_i^+-Z_i^-) &= (Q-x_{c,i})   \wedge \mathcal{I}_{c,j}
 \end{aligned}
\end{equation}
 where the impulse vectors satisfy $\mathcal{I}_{c,1}+\mathcal{I}_{c,2}=0$ by conservation of momentum.
 Proposition \ref{velocitychange}  is now a consequence of this observation and 
of Proposition \ref{prop prop}.

\section{Kinematics    of two rigid bodies}
The configuration manifold of a system of  several (unconstrained) rigid bodies in $\mathbb{R}^n$ is 
a  submanifold with boundary of the product of copies of the Euclidean group $SE(n)$, one copy for each body.
The Riemannian metric is then  the product  of  the Riemannian metrics for 
each single body. Here we focus on  the boundary of the configuration manifold of
two bodies and certain structures therein. 

Let ${B}_1$ and ${B}_2$ be
submanifolds of $\mathbb{R}^n$ of dimension $n$ having   smooth boundary and 
  equipped with mass distribution measures $\mu_1$ and
$\mu_2$ with  masses $m_j:=\mu_j({B}_j)<\infty$
and   zero first moment. 
The bodies need not be bounded.  
The configuration manifold $M$   is by definition   the closure of 
\begin{equation}\label{twobodies}M_0:=\{q=(g_1, g_2)\in G\times G: g_1(B_1)\cap g_2(B_2)=\emptyset\}\end{equation}
 where $G=SE(n)$.  We further assume that each collision configuration $q=(g_1,g_2)\in \partial M$  is such that 
 $g_1(B_1)\cap g_2(B_2)$ consists of a single point.

The definition of $M$ as the closure of $M_0$  is not  a very useful description of   $M$ near its boundary. In particular, it is not
so clear how to translate  geometric  information about the boundaries of the $B_j$ into information about the boundary of $M$. For
this purpose we introduce the  {\em extended configuration manifold}   $M_e$ defined below.

Let  $N_j$ be  the boundary of $B_j$
   and let $\nu_j$ be the outward-pointing unit normal vector field  on  $N_j$. 
 By a  (positive) {\em adapted} orthonormal frame at $b\in N_j$ of sign $\epsilon\in \{+, -\}$ we mean
 a positive orthogonal map  $\sigma:\mathbb{R}^n\rightarrow T_b\mathbb{R}^n\cong \mathbb{R}^n$  such that
$\sigma e_n= \epsilon \nu_j(b)$. Here  $e_n$ is the last vector of the standard basis $(e_1, \dots, e_n)$ of $\mathbb{R}^n$.  Hence
$\sigma$ is an  element of $SO(n)$ mapping $\mathbb{R}^n$ isometrically to $T_bN_j$.  
If  $\sigma$ is an adapted frame and $h\in H:=SO(n-1)$, then $\sigma h$ is also an adapted frame with the same base point as $\sigma$. 
In this way, $H$ acts freely and transitively by right multiplication on the set of adapted frames at any given point of $N_j$.

  \vspace{.1in}
\begin{figure}[htbp]
\begin{center}
\includegraphics[width=3.5in]{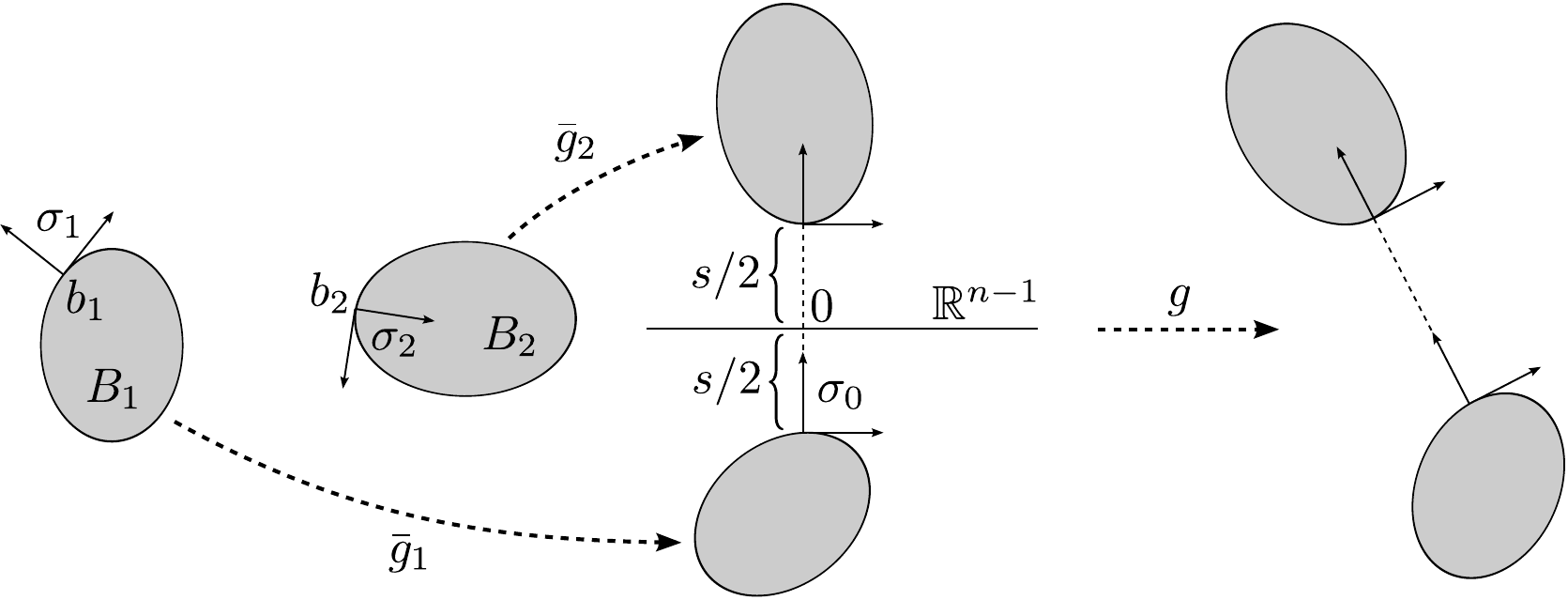}
\caption{\small  Interpretation of the map $\Psi$. The transformation $\overline{g}_j$ sends body $B_j$ from
its standard configuration to the configuration that takes the adapted frame $\sigma_j$ to the standard frame in $\mathbb{R}^n$,
and the point $b_j$ into the line through the origin along $e_n$ a distance $s/2$ from the origin.}
\label{Psi}
\end{center}
\end{figure}

  We denote by 
  $\mathcal{F}^\epsilon_j$ 
   the principal $H$-bundle of  {adapted} (positive) orthonormal frames  over $N_j$ of sign $\epsilon$.
Elements of $\mathcal{F}_j^\epsilon$ will be written  $(b, \sigma)$, where $\sigma$ is in the
fiber $\mathcal{F}^\epsilon_j(b)$.  
The {\em extended configuration manifold}  is the product
$ M_e:=\mathcal{F}^+_1 \times \mathcal{F}^-_2 \times G\times [0,\infty)$. We can now define the map
   $\Psi: M_e\rightarrow G\times G$ by
$ \Psi(b_1, \sigma_1, b_2, \sigma_2,g, s)=(g_1, g_2)$
where
\begin{align*}
g_1&=g\overline{g}_1=g\left(\sigma_1^{-1}, -\sigma_1^{-1} b_1-s/2 \right)= \left(A\sigma_1^{-1},a-\frac{s}2 Ae_n - A\sigma_1^{-1}b_1\right)\\
g_2&=g\overline{g}_2=g\left(\sigma_2^{-1}, -\sigma_2^{-1} b_2+s/2 \right)=\left(A\sigma_2^{-1},a+\frac{s}2 Ae_n-A\sigma_2^{-1} b_2\right).
\end{align*}
The geometric interpretation of  $\Psi$ is shown in Figure \ref{Psi}. Note that 
points on the boundary of $M$ correspond under $\Psi$ to points in $M_e$
with  coordinate $s=0$.  The groups $G$ and $H$ naturally act on $M_e$ on left and right, respectively:
$${g}(b_1, \sigma_1, b_2, \sigma_2,g', s)h:=(b_1, \sigma_1h, b_2, \sigma_2h,{g}g'h, s). $$
The quotient $M_e/H$ is easily seen to be a smooth manifold and the projection $M_e\rightarrow M_e/H$
is a principal $H$-bundle.  
It is also immediate from the definitions that 
$$\Psi({g} q h)={g}\Psi(q) $$
for all  $\xi\in M_e$, where the action  of $G$ on $G\times G$ is defined by ${g}(g_1, g_2)=({g}g_1,{g}g_2).$
Therefore, $\Psi$ induces a $G$-equivariant map $$\overline{\Psi}: M_e/H\rightarrow G\times G.$$  
Equivariance  means $\overline{\Psi}(gq)=g\overline{\Psi}(q)$. 
The  $G$-action on $M_e$ admits a
smooth cross-section: 
$$S:=\mathcal{F}^+_1 \times \mathcal{F}^-_2 \times \{e\}\times [0,\infty).  $$
The $G$-action on $M_e$ and on $M_e/H$ leaves invariant the coordinate $s$; in particular, it leaves the boundary of these two
manifolds invariant.

  \vspace{.1in}
\begin{figure}[htbp]
\begin{center}
\includegraphics[width=5in]{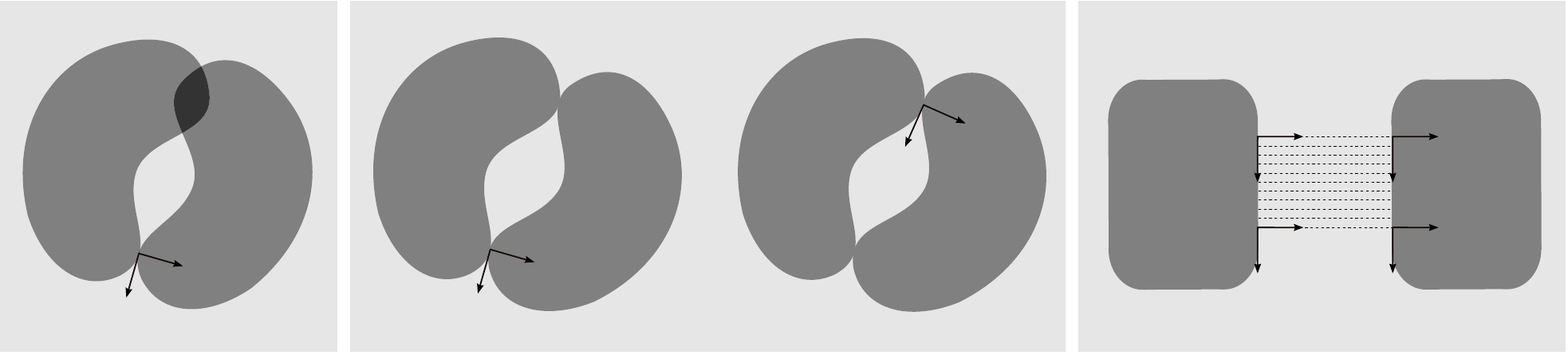}
\caption{\small  For $M_e/H$ to be  a good parametrization of $M$ near the boundary some  pathologies must 
be avoided. Far left: a boundary configuration in $M_e/H$ that  is not in   $\partial M$; middle pair:
two distinct elements of $M_e/H$ corresponding to the same element in $\partial M$; far right: a curve in $M_e$
that is mapped under $\overline{\Psi}$ to a single point in $M$.}
\label{ExamplesMe}
\end{center}
\end{figure}

It is natural to expect  that under reasonable assumptions the restriction of $\overline{\Psi}$ to a neighborhood of the boundary of $M_e/H$ will be a diffeomorphism onto a neighborhood of  the boundary of $M$, thus providing a useful parametrization 
  for the purpose of understanding collisions.  
  Figure \ref{ExamplesMe} shows some of the situations that must be avoided.
  We  give shortly a few sets of sufficient conditions for $\overline{\Psi}$
to be a local diffeomorphism, but 
our immediate goal is to explore $M_e$, $M_e/H$, and their boundaries
a little further.

\subsection{The tangent bundle of $\partial M_e$ and $\partial (M_e/H)$}
Recall that the {shape operator} of a hypersurface $N\subset \mathbb{R}^n$ with unit normal vector field $\nu$ at a point $b\in N$
is the linear map $S_b:T_bN\rightarrow T_bN$ defined by
$ v\mapsto -  D_v\nu$
 where  $D_v$ is the Levi-Civita connection
for the standard Euclidean metric in $\mathbb{R}^n$.   We write  $\nabla_v X:= \Pi_b D_v X$ for a tangent vector
field $X$ on $N$,
where   $\Pi_b$ is the orthogonal projection  from $\mathbb{R}^n$ to $T_bN$. This is 
the Levi-Civita connection on $N$ for the induced metric. 
Let $(b,\sigma)$ be a point in the adapted frame bundle $\mathcal{F}^\epsilon$  over  $N$ 
and $(v, \zeta)$ a
tangent vector
to $\mathcal{F}^\epsilon$ at $(b, \sigma)$.  Then, by differentiating $\nu(\gamma(t))=-(-1)^\epsilon\sigma(t) e_n$,
where $(\gamma(t), \sigma(t))$ is a smooth curve representing $(v,\zeta)$ at $(b,\sigma)=(\gamma(0),\sigma(0))$,
we obtain 
$$S_b(v)= (-1)^\epsilon \zeta e_n = - \zeta \sigma^{-1} \nu(b)=-\sigma V \sigma^{-1}\nu(b)$$
where $V:=\sigma^{-1} \zeta$ can be regarded as an element of $\mathfrak{so}(n)$ just as $\sigma$ is viewed as
an element of $SO(n)$. The tangent bundle of $\mathcal{F}^\epsilon$ for any smooth
hypersurface $N$  has now the following description. 
 Let $(b,\sigma)\in \mathcal{F}^\epsilon$. Then  
$$T_{(b,\sigma)}\mathcal{F}^\epsilon \cong \left\{ (v, V)\in T_bN\times \mathfrak{so}(n): S_b(v)= -\sigma V\sigma^{-1} \nu(b)  \right\}.  $$
As before, we   use the canonical  identification $TG\cong G\times \mathfrak{g}$ and  write elements of $\mathfrak{g}$ in
the form $(Z, z)\in \mathfrak{so}(n)\times \mathbb{R}^n$.  The shape operator of $N_j$ will be written $S^{(j)}$.
We omit the superscript when it is clear from the context to which   body   the operator is associated.  Then the tangent space of $M_e$ at a
point $q=(b_1, \sigma_1, b_2, \sigma_2, g, s)$ is given by
\begin{align*} T_qM_e &=\left\{(v_1, V_1, v_2, V_2, Z, z, \varrho): v_j\in T_{b_j}N_j, V_j  \in \mathfrak{so}(n), (Z,z)\in \mathfrak{g}, \varrho\in \mathbb{R},\right.\\
&\ \ \ \ \ \ \ \ \ \ \ \ \ \ \ \ \ \ \ \ \ \ \ \ \ \ \ \ \ \ \ \ \ \ \  \ \ \   \ \ \ \ \ \ \ \ \ \ \ \ \ \left. S_{b_j}v_j = -\sigma_j V_j \sigma_j^{-1} \nu_j(b_j), j=1,2 \right\}.
\end{align*}
Tangent spaces to  $\partial M_e$ consist of those vectors for which $\varrho=0$.  

 Let $G_q$ and $H_q$ represent
the orbits through $q\in \partial M_e$ of the (right and left, respectively) actions of $G$ and $H$ on $M_e$. The tangent spaces
at $q$ of the respective orbits will be written $\mathfrak{g}_q$ and $\mathfrak{h}_q$. Then
\begin{equation}\label{gh}
\mathfrak{g}_q=\left\{(0,0,0,0,Z,z,0): (Z,z)\in \mathfrak{g}\right\} \text{ and }
\mathfrak{h}_q=\left\{(0, Y, 0, Y, Y, 0, 0): Y\in \mathfrak{h}\right\}.
\end{equation}

At any $q\in \partial M_e$
the differential of $\Psi$ is 
$$d\Psi_q(v_1, V_1, v_2, V_2, Z,z,\varrho)=(Z_1, z_1, Z_2, z_2),$$
where, denoting $\text{Ad}_\sigma(W):=\sigma W\sigma^{-1}$, 
\begin{equation}\label{Zz}
\begin{aligned}
Z_j &= \text{Ad}_{\sigma_j} (Z-V_j)\\
z_j&= \sigma_jz -\text{Ad}_{\sigma_j} (Z-V_j) b_j - v_j -\varrho \nu_j(b_j).
\end{aligned}
\end{equation}

 The next proposition contains as a special case Proposition \ref{regularity}.
It uses the notation
 $$ S_j:= \sigma_j^{-1} S^{(j)}_{b_j} \sigma_j:\mathbb{R}^{n-1}\rightarrow \mathbb{R}^{n-1}$$
 for any $q=(b_1,\sigma_1, b_2, \sigma_2, g, s)$. We allow $s$ to be non-zero, in which case $S^{(j)}$ is
 the shape operator of the  level  hypersurface of $M_e/H$ corresponding to value $s$.  
 
  \vspace{.1in}
\begin{figure}[htbp]
\begin{center}
\includegraphics[width=2.5in]{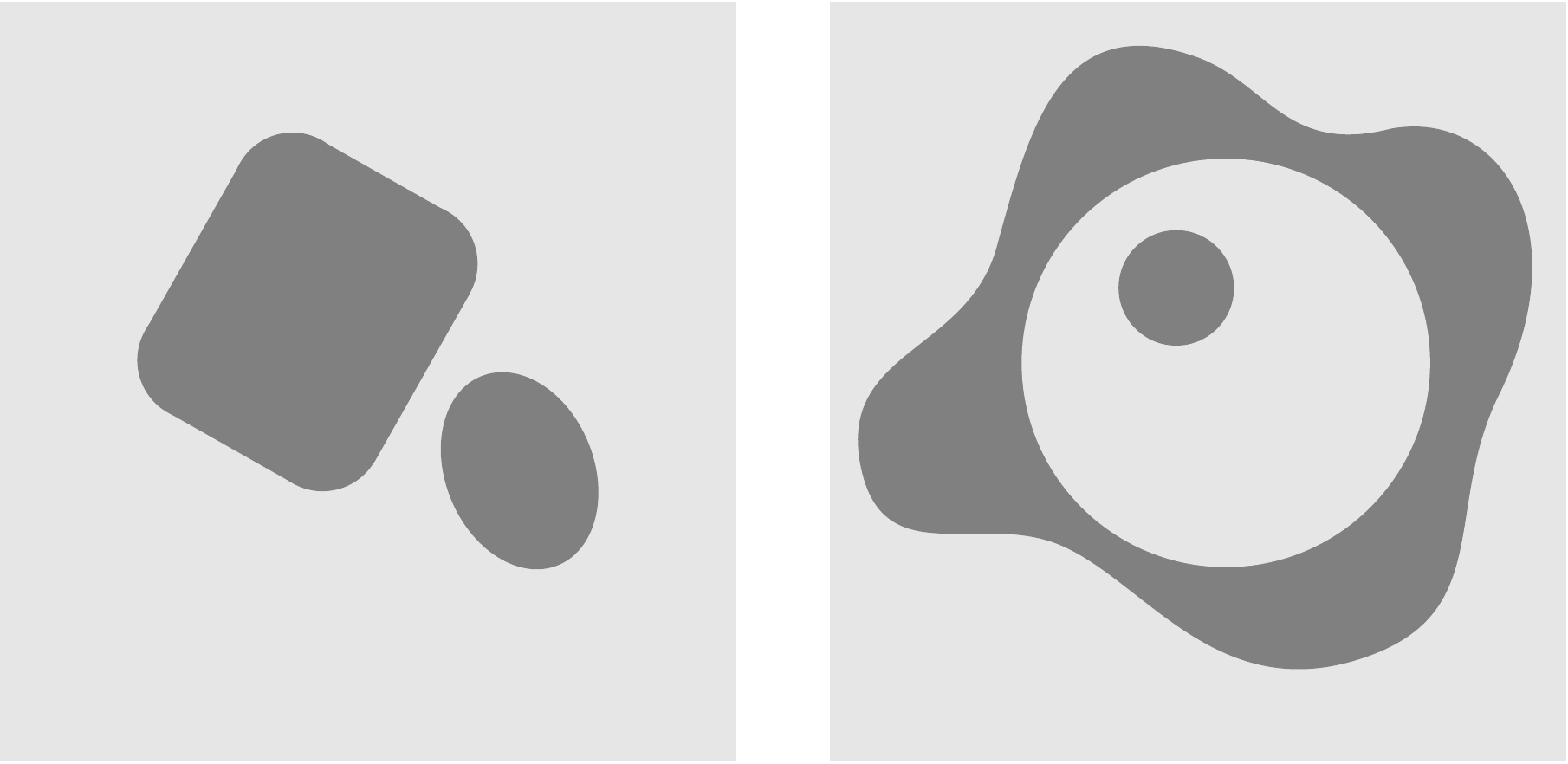}
\caption{\small  Situations for which Proposition \ref{regular} applies.}
\label{Examples}
\end{center}
\end{figure} 
 
 \begin{proposition}\label{regular}
 The map $\Psi:M_e\rightarrow G\times G$ is a submersion from a neighborhood  $\mathcal{U}$ of the boundary of $M_e$
 onto a neighborhood of the boundary of $M$ if any of the following conditions involving $\mathcal{U}$ and the shape operators
  holds.
  \begin{enumerate}
  \item  $S_1+S_2$ is non-singular at all points of $\partial M_e$. In this case, $\mathcal{U}$ is a neighborhood
  of $\partial M_e$ where this non-singular condition holds.
  \item $\mathcal{U}$ is a neighborhood of $\partial M_e$ where one of the $S_j$ is non-singular and $S_1+S_2 - s S_1S_2$ is also
  non-singular. 
  \item If the two bodies are convex and the boundary of one of them has non-vanishing Gauss-Kronecker curvature so that $S_j$ is
   everywhere non-singular on $N_j$ for some $j$, then $\mathcal{U}=M_e$.
    \end{enumerate}
 In each case $\mathcal{U}$ is $G$-invariant,  the kernel of $d\Psi_q$ is $\mathfrak{h}_q$ at each $a\in \mathcal{U}$, and $\Psi|_{\mathcal{U}}:\mathcal{U}\rightarrow \Psi(\mathcal{U})$ is a principal $H$-bundle. In addition, the boundary of $M$ is a smooth submanifold and there are smooth functions $b_j:\partial M\rightarrow  N_j$, $j=1,2$, such that  $b_1(q), b_2(q)$ are the unique points on the respective bodies that
 are brought into contact in  collision configuration $q$.
 \end{proposition}
 \begin{proof}
By counting dimensions we see that $\mathcal{U}$ should be a neighborhood of the boundary of $M_e$
where the kernel of $d\Psi_q$ is $\mathfrak{h}_q$.  
It follows from equations  \ref{gh} and \ref{Zz} above that
this kernel contains $\mathfrak{h}_q$.
We show equality under the conditions of item (2), the other cases being similar. Say that $S_2$ is non-singular. From the  explicit 
 form of    $d\Psi_q$ given in \ref{Zz}  we see that $\xi=(v_1, V_1, v_2, V_2, Z,z,\varrho)$
  lies in that kernel if and only if $Z=V_1=V_2$ and $z=\sigma_j^{-1} v_j -(-1)^j \frac{s}2 Ze_n$ for $j=1,2$.
  Observe that $\sigma_j^{-1} v_j$ and $Ze_n$ lie in $\mathbb{R}^{n-1}$, which is orthogonal to $e_n$. Hence $\varrho=0$.
  Keeping in mind  $\sigma_jV_j e_n=(-1)^jS_{b_j}v_j$, we obtain
  \begin{align*}
  -S_1 z&= Ze_n -\frac{s}{2} S_1 Ze_n\\
  S_2z&= Ze_n -\frac{s}{2} S_2 Ze_n.
  \end{align*}
 From this we conclude that $[S_1+S_2-s S_1S_2]S_2^{-1}Ze_n=0$ which, under the conditions of (2) implies that $Ze_n=0$. 
  Since $S_2$ is non-singular, this also implies that $z=0$ and $v_j=0$. Therefore, $\xi=(0,Z,0,Z,Z,0,0)$, where $Z\in \mathfrak{h}$
  since $Ze_n=0$.  That
 $\Psi|_{\mathcal{U}}$ is a principal $H$-bundle is now easy. $G$-equivariance of $\Psi$ implies that
  $\mathcal{U}$ is $G$-invariant. 
 \end{proof}

 Assuming  for simplicity that $\mathcal{U}$ of Proposition \ref{regular} is all of $M_e$ 
 (we only need what follows on some neighborhood of the boundary of $M_e$),
it is useful to know whether the principal bundle $M_e\rightarrow M$ admits a $G$-invariant connection
  since the
   associated  horizontal subspace $\mathcal{H}_q$ can then serve as a proxy for the tangent space of $T_{\bar{q}}M$,
  without having to go to the quotient.  
A principal $H$-connection on $M_e$ is given by a one-form $\omega$ taking values in  $\mathfrak{h}$
  and satisfying the properties:
  \begin{enumerate}
  \item $\omega_q(Y_q)=Y\in \mathfrak{h}$, where $Y_q$ is the   vector induced by the infinitesimal action of $\mathfrak{h}$;
  \item $h^*\omega =\text{Ad}_{h^{-1}}\circ \omega$.
  \end{enumerate}
   Let $\Pi$ be the orthogonal projection from $\mathbb{R}^n$ to $\mathbb{R}^{n-1}=e_n^\perp$.

  \begin{proposition}
  For any real constants $c_1, c_2, c_3$ such that $c_1+c_2+c_3=1$ 
the $\mathfrak{h}$-valued one-form $\omega$ on $M_e$ given by
  $$ \omega_q(v_1, V_1, v_2, V_2, Z,z,\varrho)=c_1\Pi V_1\Pi + c_2 \Pi V_2\Pi + c_3\Pi Z\Pi$$
    is a $G$-invariant $H$-connection on $M_e$. 
  \end{proposition} 
  \begin{proof}
  This is a simple check.
  \end{proof}

Let us choose $\omega_q(\xi):=\Pi V_1\Pi$ and denote by $\mathcal{H}_q$ the horizontal subspace
defined by this choice of connection form.  Recall 
the  maps $S_j$ on $\mathbb{R}^{n-1}$
 for $q=(b_1, \sigma_1, b_2, \sigma_2, g, 0)\in \partial M_e$
  given by  $S_j=\sigma_j^{-1} S_{b_j} \sigma_j$.
The  $b_j$ are determined uniquely from $\Psi(q)$ and  the $\sigma_j$ are determined uniquely up to an overall common element of $H$ acting
on the right.

\begin{proposition} Let $q=(b_1, \sigma_1, b_2, \sigma_2, g, 0)\in \partial M_e$ and suppose that 
$S_1+S_2$ is invertible. Then
  $d\Psi_q$  maps $\mathcal{H}_q$ isomorphically onto $\mathfrak{g}\times \mathfrak{g}$
  and 
  $$d\Psi_q\left(\mathcal{H}_q\cap T_q (\partial M_e)\right) = \{(Z_1,z_1, Z_2, z_2)\in \mathfrak{g}\times\mathfrak{g}: \nu_1\cdot(Z_1 b_1+z_1)+\nu_2\cdot (Z_2 b_2+z_2)=0\}$$
  where $\nu_1=\nu_1(b_1)$ and $\nu_2=\nu_2(b_2)$ are the unit normal vectors. 
\end{proposition}
\begin{proof}
The proof is elementary, but we show the main point. 
Let $\bar{\xi}=(Z_1, z_1, Z_2, z_2)\in \mathfrak{g}\times \mathfrak{g}$. We wish to show the existence of  a unique
 $\xi=(v_1, V_1, v_2, V_2, Z, z,\varrho)\in \mathcal{H}_q$ that is sent to $\bar{\xi}$ under   $d_q\Psi$.
 The components of $\xi$ satisfy: $$\sigma_j^{-1}v_j\in \mathbb{R}^{n-1},\  \Pi V_j\Pi=0,\  (Z,z)\in \mathfrak{g},\  \varrho\in \mathbb{R},\   S_j \sigma_j^{-1}v_j = (-1)^j V_j e_n.$$
 and  $Z_j, z_j$ are related to theses quantities by
 $$ Z_j=\text{Ad}_{\sigma_j}(Z-V_j), \ \ z_j= \sigma_j z-\text{Ad}_{\sigma_j}(Z-V_j)b_j -v_j -\varrho \nu_j(b_j).$$
 Writing $g=(A,a)$ and $\Psi(q)=(g_1,g_2)$, we have $g_j=(A\sigma_j^{-1}, a- A\sigma_j^{-1} b_j).$
Note that
\begin{equation}\label{v2} \sigma_1^{-1}v_1 - \sigma_2^{-1}v_2 + 2\varrho e_n  = \sigma_2^{-1}(Z_2 b_2 +z_2)- \sigma_1^{-1}(Z_1 b_1 +z_1)\end{equation}
from which we obtain $\varrho$ and $ \sigma_1^{-1}v_1 - \sigma_2^{-1}v_2\in \mathbb{R}^{n-1}$ in terms of the $Z_j$ and $z_j$. Also
 \begin{align*}
  \left(\text{Ad}_{\sigma_2^{-1}} Z_2 - \text{Ad}_{\sigma_1^{-1}}Z_1\right)e_n&=(V_2-V_2)e_n\\
 &= S_1 \sigma_1^{-1}v_1 + S_2 \sigma_2^{-1}v_2\\
 &=(S_1+S_2)\sigma_1^{-1}v_1 + S_2(\sigma_2^{-1}v_2 -\sigma_1^{-1}v_1)\\
 &=(S_1+S_2)\sigma_1^{-1}v_1 +S_2\Pi\left\{ \sigma_2^{-1}(Z_2 b_2 +z_2)- \sigma_1^{-1}(Z_1 b_1 +z_1)\right\}
 \end{align*}
 from which we obtain $\sigma_1^{-1}v_1$ in terms of the $Z_j$ and $z_j$ under the assumption that $S_1+S_2$
 is invertible. From Proposition \ref{exercisewedge}, item (8), we deduce $$V_1= \Pi V_1\Pi + e_n \wedge V_1 e_n= e_n \wedge V_1 e_n=
 -e_n\wedge S_1 \sigma_1^{-1}v_1$$
 so that 
 $V_1$ is also uniquely determined by the $Z_j$ and $z_j$. From
 $$ V_2-V_1= \text{Ad}_{\sigma_2^{-1}}Z_2 -\text{Ad}_{\sigma_1^{-1}}Z_1$$
 we obtain $V_2$ uniquely and from the above \ref{v2} we obtain $v_2$ uniquely.  From these we easily obtain
 $Z$ and $z$ as well, proving the first part of the proposition.  The second part follows   from the observation that
a vector is tangent to   $\partial M$ if and only if $\varrho=0$. 
\end{proof}
  
\subsection{The non-slipping,  rolling, and diagonal subbundles}\label{nsrd} 
Let $\gamma(t)$ be a smooth curve in $\partial M_e$ such that $q=\gamma(0)=(b_1, \sigma_1, b_2, \sigma_2, g,0)$.
We  omit the variable $s$, which is set to $0$ for a boundary point.  Let
$(\gamma_1(t), \gamma_2(t))$ be the image of $\gamma$  under  $\Psi$ and
write $\gamma_j(0)=g_j$, $\bar{q}:=(g_1, g_2)$, where $g_j=(A_j,a_j)$ and $g=(A,a)$.
Denote the components of the infinitesimal motion in $M_e$ by  $$\xi:=\gamma'(0)=
   (v_1, V_1, v_2, V_2, Z, z),$$
   omitting $\varrho=0$. 
The two bodies in configuration $q$  are in contact at $g_1(b_1)=g_2(b_2)$.  

\begin{definition}[Non-slipping and non-twisting  conditions] The infinitesimal motion $\xi \in T_q M_e$  is said to satisfy 
the {\em non-slipping condition} if
the velocities of the material points $b_j$ at the contact configuration are equal.  It is said to satisfy the {\em non-twisting condition} 
if the tangent planes to $N_j$ at $b_j$ do not rotate relative to each other under   $\xi$. 
\end{definition}

We now derive an explicit expression for these conditions. The infinitesimal motion of  $B_j$ is given by $\xi_j:=(Z_j,z_j)\in \mathfrak{g}$, 
which is obtained  from $d\Psi_q\xi$. We know that
\begin{align*}
Z_j &= \text{Ad}_{\sigma_j} (Z-V_j)\\
z_j&= \sigma_jz -\text{Ad}_{\sigma_j} (Z-V_j) b_j - v_j.
\end{align*}
Due to  Proposition \ref{Vb}
$ V_{\xi_j}(b_j)= A_j(Z_j b_j +z_j)=A(z - \sigma_j^{-1} v_j).$
The non-slipping condition,   $V_{\xi_1}(b_1)=V_{\xi_2}(b_2)$, then reduces to 
\begin{equation}
\sigma_1^{-1} v_1 = \sigma_2^{-1}v_2.
\end{equation}
Turning now to the non-twisting condition, let $u_j$ be a tangent vector to $N_j$ at $b_j$ such that,
in the contact configuration given by $q$, is sent to a common vector, for $j=1,2$, in the plane of  contact. 
Thus $A_1 u_1=A_2 u_2$. The infinitesimal rotation of $A_ju_j$ at the point of contact is 
$$ A_jZ_j u_j=A(Z-V_j)\sigma_j^{-1}u_j.$$ The orthogonal projection to the plane of contact is 
$A\Pi A^{-1}$, recalling that $\Pi$ is the orthogonal projection to $\mathbb{R}^{n-1}$.
(It may be helpful to  keep in mind Figure \ref{Psi}.)
Because $A=A_j\sigma_j$,  the non-twisting condition takes the form
$\Pi V_1\Pi=\Pi V_2 \Pi$
and since $\Pi V_1\Pi=0$ holds for  horizontal vectors, 
$\Pi V_j\Pi=0$ for $ j=1,2. $

Now let $\Psi(q)=(g_1,g_2)$, $g_j=(A_j,a_j)$ and $\bar{\xi}=(Z_1, z_1, Z_2, z_2)=d\Psi_q \xi$. The non-slipping
condition expressed in terms of $\bar{\xi}$ becomes
$$A_1[Z_1 b_1 +z_1]=A_2[Z_2b_2+z_2] $$
and the non-twisting condition becomes
$$\text{Ad}_{A_j}Z_j=W +\nu_j(b_j)\wedge w_j$$
for a  $W\in \mathfrak{so}(n)$ independent of $h$  and  $w_j\in T_{b_j}N_j$.

\begin{definition}[Non-slipping, rolling, and diagonal subbundle]\label{bundles}
The {\em non-slipping subbundle} of   $T(\partial M)$ consists of all tangent vectors
satisfying the non-slipping condition. The {\em rolling subbundle}   
of $T(\partial M)$ consists of all tangent vectors satisfying both the non-slipping and non-twisting conditions.
The {\em diagonal subbundle} of $T(\partial M)$ is the tangent bundle to the orbits of the action of $G$ on $\partial M$
defined by $g(g_1, g_2)=(gg_1, gg_2)$. We denote these three subbundles, respectively, $\mathfrak{S}, \mathfrak{R}, \mathfrak{D}$. 
We refer to these collectively as the {\em kinematic subbundles} of  $T(\partial M)$. Notice that $\mathfrak{D}_{\Psi(q)}=\mathfrak{g}_q$, using previous notation.
\end{definition}

Starting from this definition rather than  Definition \ref{kinematic bundles}, the content of the latter  becomes a statement,
which is proved by the above remarks.

\section{Collision maps}
Let now $M\subset G\times G$ be the configuration manifold of two rigid bodies in $\mathbb{R}^n$, where
$G=SE(n)$. By  condition 2 of Proposition \ref{regular}
  $M$ has smooth boundary and 
boundary points represent configurations in which the bodies are in contact at a single point.  
  Let the state of the bodies before and after collision
be given by the element of $T_qM$, $q\in \partial M$, represented by $$(Z^\pm_1, z^\pm_1, Z^\pm_2, z^\pm_2)\in \mathfrak{g}\times\mathfrak{g}.$$
Here the sign `$+$' indicates post-collision velocities and `$-$' pre-collision velocities. We  obtained in \ref{imp} 
a condition on the pre- and post-collision velocities due to  impulsive forces that act at a single point of the body.
We restate it here.
 Let  the common point  of  contact be $Q=A_jb_j+a_j$, where
$b_j$ is the material point in standard body configuration which corresponds to $Q$ in configuration $g_j=(A_j,a_j)$,
where $a_j=x_{cj}$ is the center of mass of body $j$ in the given configuration. 
 Then we obtain from expression \ref{imp}:
\begin{equation}\label{propertyfive}
 \begin{aligned}
 z_j^+  &=  z_j^-+u_j\\
 Z_j^+ &=  Z_j^- + \mathcal{L}_j^{-1}( b_j  \wedge u_j )
 \end{aligned}
 \end{equation}
 where $u_j=A^{-1}_j\mathcal{I}_{cj}/m_j$.  We should add to these equations
 $\mathcal{I}_{c1}+\mathcal{I}_{c2}=0$ for 
 conservation of (linear) momentum.

 \begin{proof}[Proof of Theorem \ref{orthogonal}]
A simple dimension count gives  $\dim \mathfrak{C}_q=n$ and  $\dim \mathfrak{S}_q = 2 \dim \mathfrak{g} -n$
 so that the sum of the two dimensions equals    $\text{dim}\, T_q M$. Therefore, it suffices to show that these subspaces are
 orthogonal.  The Riemannian metric on $M$ is the restriction of the product metric on $G\times G$ (each factor having a  possibly different metric
as the
  bodies   may have  different mass distributions.) Explicitly, let $u, v\in T_qM$ and
 write $$v=((Y_1, y_1),(Y_2, y_2)), \ w=((Z_1, z_1),(Z_2, z_2)).$$ Then
 \begin{equation}  \langle v,w\rangle_q=\sum_{j}m_j\left[\frac12\text{Tr}\left(\mathcal{L}_j(Y_j)Z_j^\dagger\right) + y_j\cdot z_j\right].
 \end{equation}
  Now consider the vectors  
\begin{align*} 
v&=((\mathcal{L}_1^{-1}(b_1\wedge u_1), u_1),(\mathcal{L}_2^{-1}(b_2\wedge u_2), u_2))\in \mathfrak{C}_q\\
w&=((A^{-1}_1 Z_1 A_1, A_1^{-1} z^*-A_1^{-1} Z_1 A_1 b_1),(A_2^{-1} Z_2 A_2, A_2^{-1} z^*-A_2^{-1} Z_2 A_2 b_2))\in \mathfrak{S}_q
 \end{align*}
 where $Z_j=Z-V_j$ and $z^*=z-z'$. 
Observe that 
 $$\text{Tr}\left((b_j\wedge u_j)(A_j^{-1} Z_j A_j)^\dagger)\right) = 2 u_j \cdot (A_j^{-1} Z_j A_j b_j). $$
Then
 \begin{align*}
 \langle v, w\rangle_q&= \sum_j m_j\left[\frac12\text{Tr}\left((b_j\wedge u_j)(A_j^{-1} Z_j A_j)^\dagger)\right) + (A_j^{-1}z^*-A_j^{-1}Z_j A_jb_j)\cdot u_j\right]\\
 &=\sum_j m_j\left[u_j \cdot (A_j^{-1} Z_j A_j b_j)+ (A_j^{-1}z^*-A_j^{-1} Z_j A_jb_j)\cdot u_j\right]\\
 &=\sum_jm_j(A_j^{-1} z^*)\cdot u_j\\
 &=z^*\cdot \sum_j m_j A_ju_j.
  \end{align*}
  But $m_1 A_1u_1+m_2 A_2u_2=0$ by  the definition of $\mathfrak{C}_q$ so
  the two vectors are orthogonal. 
 \end{proof}

 \begin{proof}[Proof of Corollary \ref{corollary}]
Let $C$ be a linear involution in $O(n-1)$. Then $C$ is diagonalizable over $\mathbb{R}$
with eigenspace decomposition 
$ \mathbb{R}^{n-1}=(C+I)\mathbb{R}^{n-1}\oplus (C-I)\mathbb{R}^{n-1}$ and eigenvalues $1, -1$ having
multiplicities $n-k-1$ and $k$, respectively, where $k\in \{0, 1, \dots, n-1\}$.
Thus for each such $C$ there is $k$ and $A\in GL(n-1,\mathbb{R})$ such that $C=A^{-1}J_k A$ where
$J_k$ is the diagonal matrix $\text{diag}(I_{n-k-1}, -I_k)$ and $I_l$ indicating the $l\times l$ identity matrix.
We can take $A$ to be orthogonal. In fact, let $A=SU$ be the polar decomposition of $A$ into a positive symmetric
part $S=\sqrt{A^\dagger A}$ and orthogonal part $U$. The condition $C^\dagger C=I$ implies that $S^2$ and $J_k$ commute,
from which it follows that $S^2$, hence $S$, is also a block matrix with $0$ on the off-diagonal blocks
of size $k\times (n-k-1)$ and $(n-k-1)\times k$.  Therefore, $S$ commutes with $J_k$ whence the claim.
Thus the set of all orthogonal involutions in dimension $n-1$ is the disjoint union of the sets $\mathcal{J}_k=\{U^\dagger J_kU: U\in O(n-1)\}$.
It is clear from this description that $\mathcal{J}_k$ is the  homogeneous space  $O(n-1)/L$, where $L$ is the isotropy group
of $J_k$.  Equivalently, $L$ is the subgroup of all $U$ that commute with $J_k$, which is easily seen to be
the product $O(n-k-1)\times O(k)$.
 \end{proof}
 
  The following proposition gives a concrete expression for the unit normal vector field. 
 \begin{proposition}
 Let $\nu_j(b_j)$ denote the unit outward pointing normal vector to body $B_j$ at the boundary point $b_j$. Then
 the unit normal vector to $\partial M$ at $q$ is given by
 $$ \mathbbm{n}_q=(c_1(\mathcal{L}_1^{-1}( b_1 \wedge    \nu_1(b_1)), \nu_1(b_1)),c_2(\mathcal{L}_2^{-1}( b_2\wedge    \nu_2(b_2)), \nu_2(b_2))) $$
 where $c_1, c_2$ are defined up to a common sign by the equations $m_1 c_1=m_2 c_2$ and
  $$\sum_j  c_j^2 m_j \left[1 + \frac12\text{\em Tr}\left(( b_j\wedge    \nu_j(b_j)  )(\mathcal{L}_j^{-1}( b_j\wedge\nu_j(b_j)))^\dagger\right)\right] =1.$$
 \end{proposition}
 \begin{proof}
 The unit normal vector $\mathbbm{n}_q$, being an element of $\mathfrak{C}_q$,  can be written
 as 
 $$ \mathbbm{n}_q=((\mathcal{L}_1^{-1}(b_1\wedge u_1), u_1),(\mathcal{L}_2^{-1}(b_2\wedge u_2), u_2))$$
 for some $u_j\in \mathbb{R}^n$.
 Recall that a vector $v=((Z_1, z_1),(Z_2, z_2))$ tangent to $\partial M$ has the 
 form
 \begin{align*}
Z_j &= \text{Ad}_{\sigma_j} (Z-V_j)\\
z_j&= \sigma_jz -\text{Ad}_{\sigma_j} (Z-V_j) b_j - v_j 
\end{align*}
where $V_j$ and $v_j$ are related through the shape operators as discussed earlier and $v_j$ is tangent to the boundary
of body $B_j$ at $b_j$. Let as before $\nu_j(b_j)$ denote the unit normal vector to body $B_j$ at $b_j$.
 Using the explicit form of the  Riemannian metric   we  
 obtain after straightforward computation that 
 $$ 0=\langle \mathbbm{n}_q, v\rangle_q=-\sum_j m_j v_j\cdot u_j. $$
 This being true for all $v_j$ implies that $u_j=c_j\nu_j(b_j)$. But $m_1\sigma_1^{-1}u_1 + m_2\sigma_2^{-1}u_2=0 $
by the definition of $\mathfrak{C}_q$ and $\sigma_j^{-1}\nu_j(b_j)=-(-1)^j e_n$. Thus the first equation. The second
equation corresponds to the condition $\|\mathbbm{n}_q\|^2=1$.
 \end{proof}
 
\section{Proof of Theorem \ref{invariance}}\label{section invariance}
Define the one-form $\theta$ on $TM$ from the kinetic energy Riemannian metric on $M$ so that 
$\theta(\xi)=\langle v, d\pi_{v} \xi\rangle$ for each $\xi\in T_{q,v}(TM)$,
where $d\pi_{v}$ is the map induced on the tangent space at $(q,v)$ of the base-point projection map $\pi:TM\rightarrow M$. 
   We briefly recall the definition of the vertical and horizontal subbundles $E^{\text{\tiny V}}$ and
$E^{\text{\tiny H}}$ of 
$T(TM)$. For simplicity of notation we denote points in $TM$ by $v$ rather than $(q,v)$. Then the fiber
$E^{\text{\tiny V}}_v$   above $v$ is the kernel of $d\pi_v$ and $E^{\text{\tiny H}}_v$ is the kernel of the {\em connection map}
$K_v:T_v(TM)\rightarrow T_qM$, defined as follows: if $\xi=w'(0)$ where $w(t)$ is a curve through $v$ representing $\xi$, then
$K_v\xi= \left.\frac{\nabla}{dt}\right|_{t=0} w(t)$. If now $X$ and $Y$ are vector fields on $TM$, then
 \begin{equation}\label{symplectic form} d\theta_v(X,Y)=\langle K_vX, d\pi_v Y\rangle -\langle K_vY, d\pi_v X\rangle.\end{equation}
See \cite{cook} for more details. 

Now let $S$ denote the boundary of $M$, $N$ the pull-back to $S$ of the tangent bundle $TM$ under the inclusion map and
for each value $\mathcal{E}>0$ define
$$N^{\mathcal{E}}:=\left\{(q,v)\in N: \frac12\|v\|^2_q=\mathcal{E}\right\}. $$
So $N^{\mathcal{E}}$ is a level set of the kinetic energy function. It is shown in \cite{cook}, that
the pull-back of $d\theta$ to $N^{\mathcal{E}}$ under the inclusion map is non-degenerate on 
$N^{\mathcal{E}}\setminus TS$, and so it defines there a symplectic form.  
If  the ambient space of the system  is  $\mathbb{R}^n$ then  $N^{\mathcal{E}}$ has dimension $2n-2$. 
The canonical billiard measure is now the measure associated to the $(2n-2)$-form $\Omega=
 (d\theta)^{n-1}$ pulled-back to 
$N^{\mathcal{E}}\setminus TS$.

The smooth field $q\mapsto \mathcal{C}_q$ of collision maps defines a smooth map (away from singularities) 
$\mathcal{C}:N^{\mathcal{E}}\rightarrow N^{\mathcal{E}}.$ The pull-back of $\theta$ under this map is easily shown to be
$$ (\mathcal{C}^*\theta)_v(\xi)=\langle v, \mathcal{C}_q d\pi_v \xi\rangle_q.$$
Note that $d\pi_v\xi\in T_qS$ whenever $\xi\in T_vN$.  Define the projections $\Pi_q^\pm$ from $T_qS$ to
the eigenspaces of $\mathcal{C}_q$  associated to eigenvalues $\pm 1$. The assumption that $\mathcal{C}$ is
parallel is equivalent to one of these projections (equivalently, both) being parallel. 
Now define $\theta^\pm_v(\xi)= \langle v, \Pi_{q}^\pm d\pi_v\xi\rangle$, so that $\theta=\theta^++\theta^-$
and $\mathcal{C}^*\theta= \theta^+-\theta^-$. Consequently,
$$ \mathcal{C}^* d\theta= d\theta^+-d\theta^-.$$
The projections $\Pi^\pm$ can also be defined on $TN^\mathcal{E}$  by requiring
$$ \Pi^\pm_q d\pi_v =d\pi_v \Pi^\pm_v, \ \ K_v \Pi^\pm_v= \Pi_q^\pm K_v.$$
Using these maps we   define  $2$-forms $\omega^\pm$ by
$\omega_v^\pm(\xi, \eta):=d\theta_v(\Pi_v^\pm \xi, \Pi_v^\pm \eta)$.  We now
wish to relate $\omega^\pm $ and $d\theta^\pm$.

First define a tensor field   $\vartheta^\pm$ on $S$ such that for $u,v\in T_qS$ and any vector fields $X, Y$ on $S$ such
that $X_q=u$ and $Y_q=v$, we have
$$\vartheta^\pm_q(u,v):=(\nabla_u\Pi^\pm)Y-(\nabla_v\Pi^\pm)X.$$
It is not difficult to verify that this is indeed a tensor field and the definition does not depend on the extensions $X, Y$ of
$u, v$. Furthermore, $\vartheta^\pm$ vanishes under the conditions of Theorem \ref{invariance}.
A straightforward calculation now shows that 
$$ d\theta_v^\pm(\xi, \eta)= \omega^\pm_v(\xi,\eta)+ \langle v, \vartheta^\pm(d\pi_v\xi, d\pi_v \eta)\rangle.$$
Therefore, $d\theta^\pm=\omega^\pm$ when the field of collision maps is parallel.  Moreover, $d\theta= \omega^+ +\omega^-$
and $\mathcal{C}^*\omega^\pm =\pm \omega^\pm. $ It is now easy to check
that
$$(d\theta)^{n-1}= \pm (\omega^+)^{n_+}\wedge(\omega^-)^{n_-}$$
       where $n_\pm$ are the dimensions of the eigenspaces of $\mathcal{C}_q$ associated to eigenvalue $\pm 1$, and
       we finally obtain
$\mathcal{C}^*(d\theta)^{n-1}=\pm (d\theta)^{n-1}.$ Therefore, the measure induced by $(d\theta)^{n-1}$ is invariant under $\mathcal{C}$.
       
\bibliographystyle{amsalpha}

\end{document}